\documentclass[a4paper,11pt]{article}
\usepackage[latin1]{inputenc}
\usepackage[english]{babel}
\usepackage{amsmath}
\usepackage{amsfonts}
\usepackage{amssymb}
\usepackage{epsfig}
\usepackage{amsopn}
\usepackage{amsthm}
\usepackage{color}
\usepackage{graphicx}
\usepackage{subfigure}
\usepackage{enumerate}
\newtheorem{theorem}{Theorem}[section]

\newtheorem{lemma}[theorem]{Lemma}
\newtheorem{proposition}[theorem]{Proposition}
\newtheorem{definition}[theorem]{Definition}

\newtheorem*{theorem*}{Theorem}
\newtheorem*{lemma*}{Lemma}
\newtheorem*{remark*}{Remark}
\newtheorem*{definition*}{Definition}
\newtheorem*{proposition*}{Proposition}
\newtheorem*{corollary*}{Corollary}
\numberwithin{equation}{section}
\newcommand{\real}{\mathbb{R}}

\def\qed{\,\unskip\kern 6pt \penalty 500
\raise -2pt\hbox{\vrule \vbox to8pt{\hrule width 6pt
\vfill\hrule}\vrule}\par}
\definecolor{darkblue}{rgb}{0.05, .05, .65}
\definecolor{darkgreen}{rgb}{0.1, .65, .1}
\definecolor{darkred}{rgb}{0.8,0,0}
\newcommand{\beqn}{\begin{equation}}
\newcommand{\eeqn}{\end{equation}}
\newcommand{\bear}{\begin{eqnarray}}
\newcommand{\eear}{\end{eqnarray}}
\newcommand{\bean}{\begin{eqnarray*}}
\newcommand{\eean}{\end{eqnarray*}}
\begin{document}

\title{\huge \bf Self-similar blow-up patterns for a reaction-diffusion equation with weighted reaction in general dimension}

\author{
\Large Razvan Gabriel Iagar\,\footnote{Departamento de Matem\'{a}tica
Aplicada, Ciencia e Ingenieria de los Materiales y Tecnologia
Electr\'onica, Universidad Rey Juan Carlos, M\'{o}stoles,
28933, Madrid, Spain, \textit{e-mail:} razvan.iagar@urjc.es},\\
[4pt] \Large Ana I. Mu\~{n}oz\,\footnote{Departamento de Matem\'{a}tica
Aplicada, Ciencia e Ingenieria de los Materiales y Tecnologia
Electr\'onica, Universidad Rey Juan Carlos, M\'{o}stoles,
28933, Madrid, Spain, \textit{e-mail:} anaisabel.munoz@urjc.es},
\\[4pt] \Large Ariel S\'{a}nchez,\footnote{Departamento de Matem\'{a}tica
Aplicada, Ciencia e Ingenieria de los Materiales y Tecnologia
Electr\'onica, Universidad Rey Juan Carlos, M\'{o}stoles,
28933, Madrid, Spain, \textit{e-mail:} ariel.sanchez@urjc.es}\\
[4pt] }
\date{}
\maketitle

\begin{abstract}
We classify the finite time blow-up profiles for the following reaction-diffusion equation with unbounded weight:
$$
\partial_tu=\Delta u^m+|x|^{\sigma}u^p,
$$
posed in any space dimension $x\in\real^N$, $t\geq0$ and with exponents $m>1$, $p\in(0,1)$ and $\sigma>2(1-p)/(m-1)$. We prove that blow-up profiles in backward self-similar form exist for the indicated range of parameters, showing thus that the unbounded weight has a strong influence on the dynamics of the equation, merging with the nonlinear reaction in order to produce finite time blow-up. We also prove that all the blow-up profiles are \emph{compactly supported} and might present two different types of interface behavior and three different possible \emph{good behaviors} near the origin, with direct influence on the blow-up behavior of the solutions. We classify all these profiles with respect to these different local behaviors depending on the magnitude of $\sigma$. This paper generalizes in dimension $N>1$ previous results by the authors in dimension $N=1$ and also includes some finer classification of the profiles for $\sigma$ large that is new even in dimension $N=1$.
\end{abstract}

\

\noindent {\bf Mathematics Subject Classification 2020:} 35A24, 35B44, 35C06,
35K10, 35K65.

\smallskip

\noindent {\bf Keywords and phrases:} reaction-diffusion equations,
weighted reaction, blow-up, self-similar solutions, phase
space analysis, strong reaction

\section{Introduction}

The aim of this paper is to classify the blow-up profiles in self-similar form for the following reaction-diffusion equation
\begin{equation}\label{eq1}
\partial_tu=\Delta u^m+|x|^{\sigma}u^p
\end{equation}
posed for $(x,t)\in\real^N\times(0,\infty)$, in the range of exponents
\begin{equation}\label{range.exp}
m>1, \ \quad 0<p<1, \quad \sigma>\frac{2(1-p)}{m-1},
\end{equation}
the most interesting feature of this equation being the presence of an unbounded weight on the reaction term, which, as we shall see, has a very strong influence on the dynamics of it. In particular, the lower bound we impose $\sigma>2(1-p)/(m-1)$ is the optimal one leading to the phenomenon of finite time blow-up, which does not occur either for the non-weighted (homogeneous) equation with $\sigma=0$ studied in \cite{dPV90, dPV91, dPV92} or even for the limiting case $\sigma=2(1-p)/(m-1)$, a case studied recently in \cite{IS21Eternal} where it is shown that eternal solutions in self-similar exponential form exist. Here and throughout the paper, by finite time blow-up we understand the case when a solution becomes unbounded at some time $T\in(0,\infty)$. More rigorously, we say that a solution $u$ to \eqref{eq1} blows up in finite time if there exists $T\in(0,\infty)$ such that $u(T)\not\in L^{\infty}(\real)$, but $u(t)\in L^{\infty}(\real)$ for any $t\in(0,T)$. The smallest time $T<\infty$ satisfying this property is called the blow-up time of $u$. We will also use frequently the following notation: $u(t)$ designs the map $x\mapsto u(x,t)$ for a fixed time $t\in[0,T]$, and $u_t$ designs the derivative with respect to the time variable.

Studying the effects of the \emph{competition} between the diffusion term and the presence of an \emph{unbounded} weight on the reaction term for reaction-diffusion equations has been an interesting research subject in the last three decades. Chronologically, the first results deal with the semilinear case (that is, letting $m=1$ in \eqref{eq1}) and with reaction exponents $p>1$, and a number of well-known mathematicians were involved, such as Bandle, Levine, Baras, Kersner, Pinsky, DiBenedetto, Andreucci et. al., see for example works such as \cite{BK87, BL89, Pi97, Pi98} and references therein. In these works, the main question was for which initial conditions $u_0$ finite time blow-up of the solution $u(t)$ occurs. A different approach is given in \cite{AdB91} for any $m\geq1$, where a qualitative analysis of solutions and their regularity, including Harnack inequalities and many other functional properties of them, is performed for $p>1$ and with weights of the form $(1+|x|)^{\sigma}$ for any signs of $\sigma$, that is, both bounded and unbounded. In the latter work, the case $p\in(0,1)$ is left aside due to its functional difficulty. Let us also stress here that the qualitative analysis of solutions when non-constant and unbounded coefficients appear is very difficult, since a number of standard results with respect to regularity and comparison do not hold true for such weights.

More recently, Suzuki \cite{Su02} established rather sharp conditions on the exponent $p>m$ and the behavior of the initial condition $u_0(x)$ as $|x|\to\infty$ for finite time blow-up to occur, respectively for global solutions to exist, introducing a Fujita-type exponent for the weighted quasilinear equation and the second critical exponent corresponding to the behavior of the initial condition. In the same range $m>1$ and $p>m$, Andreucci and Tedeev established in \cite{AT05} a blow-up rate of general solutions, provided $\sigma$ is limited by some upper bound which, in our opinion, is a technical limitation. Later on, a series of paper by Guo, Shimojo, Souplet \cite{GLS10, GS11, GLS13, GS18} gave an answer in the semilinear case $m=1$ to the very interesting question of whether the point $x=0$ (or more generally, any zero of a weight $V(x)$ in front of the reaction term) can be a blow-up point. There are also available results on the dynamics and blow-up properties for reaction-diffusion equations with localized weights, such as \cite{FdPV06} in dimension $N=1$ and \cite{KWZ11, Liang12, FdP18} for $N\geq2$, in which the weight $|x|^{\sigma}$ is replaced by a localized, compactly supported and bounded weight $a(x)$. It is there shown that the solutions may sometimes be global and present grow-up instead of blow-up, depending on the support of $a(x)$. These results enlighten the significance of the unboundedness of the weight for the finite time blow-up to occur for Eq. \eqref{eq1}.

In view of these precedents, we started a long-term project of understanding the effects of the unbounded weight $|x|^{\sigma}$ on the dynamics of Eq. \eqref{eq1}. Since it is already well established that, in nonlinear diffusion and reaction-diffusion equations, the \emph{self-similar solutions} represent both examples of solutions to the equation and patterns for the large-time behavior of the general solutions (see the excellent monograph \cite{VPME} for the utmost importance of the self-similarity in the study of such equations), our first goal is to classify the self-similar profiles in terms of their behavior (that will be made more precise below). While a detailed analysis of self-similarity for the case $\sigma=0$ is available in \cite[Chapter 4]{S4}, previous works on different ranges of the exponents $m$, $p$ and $\sigma>0$ led to some completely novel and, in many occasions, unexpected forms and behaviors of the profiles, see for example \cite{IS19} (for $p=1$ and $\sigma>0$), \cite{IS21a} (for $p\in(1,m)$ and $\sigma>0$) or \cite{IS20a, IS21c} (for $p=m$ and dimensions $N=1$, respectively $N\geq2$), all them for the range $m>1$. A separate mention has to be given to the works \cite{IS20b, IS21b}, dealing with our range of exponents \eqref{range.exp} but for dimension $N=1$, where we classify the self-similar blow-up profiles and obtain that there are two different types of compactly supported profiles, depending on the interface behavior and also on the sign of the number $m+p-2$, a fact that is in line with the more qualitative study performed in the series of papers by DePablo and V\'azquez \cite{dPV90, dPV91, dPV92} for the homogeneous equation, that is, Eq. \eqref{eq1} with $\sigma=0$. We mention here also our recent work \cite{IS21Eternal} in which the limits of the blow-up behavior are explored: it is proved there that for $\sigma=2(1-p)/(m-1)$, \emph{eternal solutions} with exponentially fast grow-up in time exist and finite time blow-up is not expected to occur for this limiting exponent. We thus deduce that the range of $\sigma$ we consider is \emph{optimal} with respect to finite time blow-up, and of course this illustrates the strong influence of the magnitude of $\sigma$ on the dynamics of Eq. \eqref{eq1}.

We are now in a position to present our main results.

\medskip

\noindent \textbf{Main results.} Since, following the previous discussion, we expect that finite time blow-up will hold true in the range of exponents \eqref{range.exp}, we look for \emph{backward self-similar solutions} in the general form
\begin{equation}\label{BSS}
u(x,t)=(T-t)^{-\alpha}f(\xi), \qquad \xi=|x|(T-t)^{\beta},
\end{equation}
where $T\in(0,\infty)$ is the blow-up time. Plugging the ansatz \eqref{BSS} into Eq. \eqref{eq1}, it is easy to get that
\begin{equation}\label{selfsim.exp}
\alpha=\frac{\sigma+2}{\sigma(m-1)+2(p-1)}, \qquad \beta=\frac{m-p}{\sigma(m-1)+2(p-1)},
\end{equation}
while the self-similar profile $f(\xi)$ is obtained as a solution to the following ordinary differential equation
\begin{equation}\label{SSODE}
(f^m)''(\xi)+\frac{N-1}{\xi}(f^m)'(\xi)-\alpha f(\xi)+\beta\xi f'(\xi)+\xi^{\sigma}f(\xi)^p=0, \ \ \xi\in[0,\infty).
\end{equation}
Since the self-similarity exponents $\alpha$ and $\beta$ are fixed, we only need to determine the profiles of the solutions, thus our classification reduces to a study of Eq. \eqref{SSODE}. We will employ a \emph{phase-space analysis} technique in order to achieve this goal, but before stating the rigorous results and entering the details we define below what we understand by a \emph{good profile}.
\begin{definition}\label{def1}
We say that a solution $f$ to the differential equation \eqref{SSODE} is a \textbf{good profile} if it fulfills one of the following three types of behavior at $\xi=0$: either
\begin{equation}\label{beh.Q1}
f(0)>0, \qquad f(\xi)\sim\left[C+\frac{\alpha(m-1)}{2mN}\xi^2\right]^{1/(m-1)}, \ C>0, \qquad {\rm as} \ \xi\to0.
\end{equation}
or
\begin{equation}\label{beh.P2}
f(0)=0, \qquad f(\xi)\sim\left[\frac{m-1}{2m(mN-N+2)}\right]^{1/(m-1)}\xi^{2/(m-1)}, \qquad {\rm as} \ \xi\to0,
\end{equation}
or as a third possibility
\begin{equation}\label{beh.P0}
f(0)=0, \qquad f(\xi)\sim K\xi^{(\sigma+2)/(m-p)}, \ \ K>0, \qquad {\rm as} \ \xi\to0.
\end{equation}

\noindent We say that a profile $f$ has an \textbf{interface} at some point $\xi_0\in(0,\infty)$ if
$$
f(\xi_0)=0, \qquad (f^m)'(\xi_0)=0, \qquad f>0 \ {\rm on} \ (\xi_0-\delta,\xi_0), \ {\rm for \ some \ } \delta>0.
$$
\end{definition}
The local behaviors given by \eqref{beh.Q1}, \eqref{beh.P2} and \eqref{beh.P0} as $\xi\to0$ are the only possible ones for self-similar profiles, as it will follow from the forthcoming analysis. The condition $(f^m)'(\xi_0)=0$ is the standard flow condition ensuring that the self-similar solution \eqref{BSS} is a weak solution in a neighborhood of the interface point (see \cite[Section 9.8]{VPME}). We are interested in the classification of \emph{good profiles with interface} and, similarly as in dimension $N=1$ \cite{IS20b}, a profile $f(\xi)$ can form two different types of interface at a positive point $\xi_0\in(0,\infty)$, these are

$\bullet$ $f(\xi)\sim(\xi_0-\xi)^{1/(m-1)}$, as $\xi\to\xi_0$, is the standard interface behavior inherited from the diffusion term, as we recall that the Barenblatt solutions to the standard porous medium equation have this type of contact at the interface. This is called \emph{interface of} \textbf{Type I} throughout the text.

$\bullet$ $f(\xi)\sim(\xi_0-\xi)^{1/(1-p)}$, as $\xi\to\xi_0$ is a different interface behavior that is characteristic to reaction exponents $p<1$ and will be analyzed in the paper. We will call it \emph{interface of} \textbf{Type II} throughout the text.

Notice that, for $m+p=2$, the two local interface behaviors coincide and reduce to a single one. Moreover, also for the case $m+p=2$, where the condition on $\sigma$ in \eqref{range.exp} reduces to the simpler lower bound $\sigma>2$, following \cite{IS21b} we let
\begin{equation}\label{ximax}
\xi_{max}:=\left(\frac{\beta^2}{4m}\right)^{1/(\sigma-2)}=\left(\frac{1}{m(\sigma-2)^2}\right)^{1/(\sigma-2)}\in(0,\infty),
\end{equation}
and find that the localization of possible interface points of profiles only to points $\xi_0\in(0,\xi_{max}]$ remains true also in dimension $N\geq2$. We can thus state our main existence result for blow-up profiles.
\begin{theorem}\label{th.exist}
Let $m$, $p$ and $\sigma$ be as in \eqref{range.exp}. Then
\begin{itemize}
  \item if $m+p>2$, then there exists at least one good profile (in the sense of Definition \ref{def1}) with interface of Type I and one good profile with interface of Type II to Eq. \eqref{SSODE}.
  \item if $m+p=2$, for any $\xi_0\in(0,\xi_{max}]$, there exists at least a good profile (in the sense of Definition \ref{def1}) with interface to Eq. \eqref{SSODE}  having its interface point at $\xi=\xi_0$. For any $\xi_0>\xi_{max}$ there is no good profile with interface at $\xi_0$.
  \item if $m+p<2$ there are no good profiles with interface to Eq. \eqref{SSODE}.
\end{itemize}
\end{theorem}
Let us notice here that for $m+p=2$ we do not specify what kind of interface the profile might have at $\xi=\xi_0$, since, as mentioned above, the two types of interface coincide in this case. However, it is a very interesting point to classify the good profiles with respect to their behavior as $\xi\to0$, given the three possibilities in Definition \ref{def1}.

\medskip

\noindent \textbf{Classification of the profiles}. This is the finest point of the paper, which is partly analogous and partly improves (even for $N=1$) the classification performed in the two works \cite{IS20b, IS21b} devoted to the problem in dimension $N=1$. Before stating the results, let us stress that this classification has an immediate consequence on the blow-up set of the solutions and in particular on the answer to the question whether $x=0$ (where at formal level there is no reaction) can be a blow-up point. More precisely, we observe that

$\bullet$ the solutions $u(x,t)$ given by \eqref{BSS} with profiles $f(\xi)$ behaving as in \eqref{beh.Q1} \emph{blow up simultaneously} at every point $x\in\real^N$:
\begin{equation*}
u(x,t)=(T-t)^{-\alpha}f(|x|(T-t)^{\beta})\sim f(0)(T-t)^{-\alpha}, \quad {\rm as} \ t\to T.
\end{equation*}

$\bullet$ the solutions $u(x,t)$ given by \eqref{BSS} with profiles $f(\xi)$ behaving as in \eqref{beh.P2} also \emph{blow up simultaneously} at every point $x\in\real^N$ but \emph{with different rate}:
\begin{equation*}
u(x,t)\sim K(T-t)^{-\alpha+2\beta/(m-1)}|x|^{2/(m-1)}=K(T-t)^{-1/(m-1)}|x|^{2/(m-1)}, \quad {\rm as} \ t\to T.
\end{equation*}

$\bullet$ the solutions $u(x,t)$ given by \eqref{BSS} with profiles $f(\xi)$ behaving as in \eqref{beh.P0} blow up \emph{only at space infinity}, in the sense of \cite{La84, GU06}: $\|u(t)\|_{\infty}\to\infty$ as $t\to T$ but the maximum is attained on curves $x(t)$ depending on $t$ such that $|x(t)|\to\infty$ as $t\to T$, as shown by the following estimate for any $x\in\real^N$ fixed:
$$
u(x,t)\sim C(T-t)^{-\alpha+(\sigma+2)\beta/(m-p)}|x|^{(\sigma+2)/(m-p)}=C|x|^{(\sigma+2)/(m-p)}<\infty, \quad {\rm as} \ t\to T.
$$

The classification of the profiles depends strongly on the magnitude of $\sigma$. We state first the classification for $\sigma$ sufficiently close to its lower bound $2(1-p)/(m-1)$, always with $m+p\geq2$, as otherwise there are no good profiles at all.
\begin{theorem}\label{th.small}
Let $m$, $p$ and $\sigma$ be as in \eqref{range.exp}. We then have:
\begin{enumerate}
  \item For any $\sigma>2(1-p)/(m-1)$, there exist good profiles with interface of Type II and with behavior given by \eqref{beh.P0} as $\xi\to0$.
  \item There exists $\sigma_0\in(2(1-p)/(m-1),\infty)$ such that for any $\sigma\in(2(1-p)/(m-1),\sigma_0)$, \textbf{all} good profiles behaving as in \eqref{beh.P0} and \eqref{beh.P2} as $\xi\to0$ present an interface of Type II. Moreover, in this same range of $\sigma$, there exist good profiles with interface of Type I whose behavior is given by \eqref{beh.Q1} as $\xi\to0$.
  \item If $m+p>2$, there exists $\sigma_*\in(2(1-p)/(m-1),\infty)$ such that for $\sigma=\sigma_*$ there exists a unique profile $f_*(\xi)$ with behavior as in \eqref{beh.P2} as $\xi\to0$ and with interface of Type I.
  \item If $m+p=2$, there exists $\sigma_0>2$ such that for $\sigma\in(2,\sigma_0)$, there exist good profiles with interface having any of the three possible behaviors as $\xi\to0$ in Definition \ref{def1}.
  \item If $m+p=2$, then for any $\xi_0\in(0,\xi_{max}]$ there exists $\sigma(\xi_0)>2$ (depending on $\xi_0$) such that for $\sigma=\sigma(\xi_0)$ there exists a good profile with interface at $\xi=\xi_0$ and behavior as in \eqref{beh.P2} as $\xi\to0$.
\end{enumerate}
\end{theorem}
We thus find that, for $\sigma$ sufficiently close to its lower bound in \eqref{range.exp}, there are in fact good profiles with interface of all three possible types with respect to their local behavior as $\xi\to0$. Moreover, in the specific case $m+p=2$, we can even fix the interface point and get this result. This is in sharp contrast to higher values of the exponent $\sigma$, where only a single type of behavior may exist, as it readily follows from the following statement.
\begin{theorem}\label{th.large}
Let $m$, $p$ and $\sigma$ be as in \eqref{range.exp}. Then there exists $\sigma_1$ sufficiently large such that for any $\sigma\in(\sigma_1,\infty)$, there are no good profiles with interface and local behavior as $\xi\to0$ given by either \eqref{beh.Q1} or \eqref{beh.P2}. Thus \textbf{all} the good profiles with interface (of both types) behave as in \eqref{beh.P0} as $\xi\to0$.
\end{theorem}
In particular, we notice that the answer to the question of whether $x=0$ can be a blow-up point or not depends strongly on $\sigma$: for $\sigma$ relatively close to the lower limit given in \eqref{range.exp} this may happen, while for $\sigma>\sigma_1$ the point $x=0$ is not a blow-up point at least for self-similar solutions, and \emph{all} such solutions blow up at space infinity. Let us notice that Theorem \ref{th.large} brings an \emph{improvement which is new also in dimension $N=1$} with respect to \cite{IS20b, IS21b}, namely, the fact that also \eqref{beh.Q1} is ruled out for $\sigma$ large. This is very significant, as it limits to one single type the blow-up patterns and the type of blow-up expected in general for such big $\sigma$. Let us also remark that the two interface behaviors are completely different, as the expansion of the supports of the solutions obeys two different \emph{interface equations}, as discussed in the Introduction to \cite{IS20b}. We omit here this discussion, which is very similar to the one in the mentioned work devoted to dimension $N=1$.

\medskip

\noindent \textbf{Structure of the paper}. As explained above, the main technique we employ for the proofs is the \emph{phase-space analysis}, extending and adapting the dynamical system used for dimension $N=1$ in \cite{IS20b, IS21b}. The description and local analysis of the dynamical system we use is split into two parts: a Section \ref{sec.local} devoted to the analysis of the finite critical points and a Section \ref{sec.inf} devoted to the analysis of the infinite critical points. The latter will depart from the analysis in dimension $N=1$. The critical dimension $N=2$ gives rise to what is known as a \emph{transcritical bifurcation} in the phase space, thus the analysis is also split between the cases $N\geq3$ and $N=2$. The proof (very similar to the one-dimensional case) of the existence Theorem \ref{th.exist} is given in Section \ref{sec.exist}. The most specialized and deepest part of the paper, dealing with the classification of the profiles, will be split between Section \ref{sec.small}, where Theorem \ref{th.small} is proved, with a different approach for exponents satisfying $m+p=2$, and Section \ref{sec.large} dedicated to the classification of the profiles for $\sigma$ large, where there is no need to consider the case $m+p=2$ separately. This last section includes the proof of Theorem \ref{th.large}, which follows a new geometrical construction leading to an important improvement also in dimension $N=1$ over previous results published in \cite{IS20b, IS21b}.

\section{The phase space. Local analysis}\label{sec.local}

The main tool in the proofs of the main results of the present work is the analysis of a phase space associated to a quadratic dynamical system which is equivalent to the non-autonomous equation \eqref{SSODE}. Following the same change of variable as in dimension $N=1$, we transform Eq. \eqref{SSODE} into an autonomous system by letting
\begin{equation}\label{PSvar1}
X(\eta)=\frac{m}{\alpha}\xi^{-2}f^{m-1}(\xi), \ Y(\eta)=\frac{m}{\alpha}\xi^{-1}f^{m-2}(\xi)f'(\xi), \ Z(\eta)=\frac{m}{\alpha^2}\xi^{\sigma-2}f^{m+p-2}(\xi),
\end{equation}
where we recall that $\alpha$ (and also $\beta$) is defined in \eqref{selfsim.exp} and the new independent variable $\eta=\eta(\xi)$ is defined through the differential equation
$$
\frac{d\eta}{d\xi}=\frac{\alpha}{m}\xi f^{1-m}(\xi).
$$
The differential equation \eqref{SSODE} transforms into a system where the influence of the new dimension $N$ appears in the second equation:
\begin{equation}\label{PSsyst1}
\left\{\begin{array}{ll}\dot{X}=X[(m-1)Y-2X],\\
\dot{Y}=-Y^2-\frac{\beta}{\alpha}Y+X-NXY-Z,\\
\dot{Z}=Z[(m+p-2)Y+(\sigma-2)X].\end{array}\right.
\end{equation}
Notice first that the planes $\{X=0\}$ and $\{Z=0\}$ are invariant and that $X\geq0$, $Z\geq0$, which is in fact obvious from their definitions in \eqref{PSvar1}, only $Y$ being allowed to change sign. If $m+p>2$, we readily find that there are three critical points in the finite plane:
$$
P_0=(0,0,0), \ P_1=\left(0,-\frac{\beta}{\alpha},0\right), \ {\rm and} \ P_2=\left(\frac{m-1}{2\alpha(mN-N+2)},\frac{1}{\alpha(mN-N+2)},0\right).
$$
An interesting change appears in the case $m+p=2$, when the third equation simplifies and allows for new critical points. Indeed, by letting $X=0$ we already vanish both first and third equation in \eqref{PSvar1}, thus being left with a full \emph{critical parabola}, that is
$$
P_0^{\lambda}=\left(0,\lambda,-\lambda^2-\frac{\beta}{\alpha}\lambda\right), \qquad \lambda\in\left[-\frac{\beta}{\alpha},0\right]
$$
connecting the ancient critical points $P_0$ and $P_1$. This critical parabola, also considered in dimension $N=1$ in \cite{IS21b}, introduces important differences in the analysis of the special case $m+p=2$.

\subsection{Local analysis for $m+p>2$}\label{subsec.bigger}

The \textbf{critical point $P_0$} is the most complicated one to study, as it is non-hyperbolic. Indeed, the linearization of the system \eqref{PSsyst1} in a neighborhood of $P_0$ has the matrix
$$
M(P_0)=\left(
      \begin{array}{ccc}
        0 & 0 & 0 \\
        1 & -\frac{\beta}{\alpha} & -1 \\
        0 & 0 & 0 \\
      \end{array}
    \right)
$$
with a one-dimensional stable manifold and a two-dimensional center manifold.
\begin{lemma}[Local analysis near $P_0$]\label{lem.P0}
Let $m$, $p$ and $\sigma$ as in \eqref{range.exp} and such that $m+p>2$. Then the critical point $P_0$ presents an elliptic sector and a hyperbolic sector. The orbits contained in the elliptic sector go out and enter $P_0$ and contain profiles $f(\xi)$ with behavior given by \eqref{beh.P0} as $\xi\to0$ and with an interface of Type II.
\end{lemma}
\begin{proof}
We have to study the orbits contained in the center manifold near $P_0$. To this end, we perform the change of variable
$$
T:=\frac{\beta}{\alpha}Y-X+Z
$$
and after rather tedious but direct calculations, the system in variables $(X,T,Z)$ becomes
\begin{equation}\label{interm1}
\left\{\begin{array}{ll}\dot{X}&=\frac{1}{\beta}X[X+(m-1)\alpha T-(m-1)\alpha Z],\\
\dot{T}&=-\frac{\beta}{\alpha}T-\frac{\alpha}{\beta}T^2-\frac{\alpha(m+1)+N\beta}{\beta}XT-\frac{\alpha(m+p)}{\beta}TZ\\
&-\frac{m\alpha+(N-2)\beta}{\beta}X^2+\frac{(N+2)\beta+2\alpha+3}{\beta}XZ-\frac{\alpha(m+p-1)}{\beta}Z^2,\\
\dot{Z}&=\frac{1}{\beta}Z[2X+(m+p-2)\alpha T-(m+p-2)\alpha Z].\end{array}\right.
\end{equation}
We next apply the local center manifold theorem \cite[Theorem 1, Section 2.12]{Pe} to obtain the equation of the center manifold $T(X,Z)$, which after neglecting the third order terms in the equation of $\dot{T}$ in the system \eqref{interm1} becomes
$$
T=\frac{\alpha}{\beta}\left[-\frac{m\alpha+(N-2)\beta}{\beta}X^2+\frac{(N+2)\beta+2\alpha+3}{\beta}XZ-\frac{\alpha(m+p-1)}{\beta}Z^2\right]+O(|(X,Z)|^3)
$$
and the flow on the center manifold is given by the approximately homogeneous system
\begin{equation}\label{interm2}
\left\{\begin{array}{ll}\dot{X}&=\frac{1}{\beta}X[X-(m-1)\alpha Z]+O(|(X,Z)|^3),\\
\dot{Z}&=\frac{1}{\beta}Z[2X-(m+p-2)\alpha Z]+O(|(X,Z)|^3).\end{array}\right.
\end{equation}
Let us notice that the influence of the dimension $N$ in the system \eqref{interm2} is not essential, being hidden in the third order terms. Thus, this system in a neighborhood of the point $(X,Z)=(0,0)$ is well approximated by the same homogeneous system studied in detail in \cite[Section 2]{IS20b} (to which we refer the reader) with the help of the rather complicated but very useful full classification of the $(2,2)$-homogeneous dynamical systems established by Date in his well known paper \cite{Date79}. The calculations performed in \cite[Section 2]{IS20b} give that the phase portrait of the system \eqref{interm2} corresponds to the portrait no. 8 in \cite[Figure 8, p. 329]{Date79}. We thus conclude that the system \eqref{interm2} presents an elliptic sector in a sufficiently small neighborhood of the origin (together with a complementary hyperbolic one consisting in orbits that only go out of $P_0$), hence there are orbits in the phase space associated to the initial system \eqref{PSsyst1} which first go out and then return to the critical point $P_0$ along the center manifold. The local behavior of the profiles contained in these orbits in a neighborhood of the critical point $P_0$ is given by $T\sim0$, or equivalently, taking into account the definitions of $X$, $Y$, $Z$ in \eqref{PSvar1},
\begin{equation}\label{interm3}
f'(\xi)-\frac{\alpha}{\beta}\xi^{-1}f(\xi)+\frac{1}{\beta}\xi^{\sigma-1}f^{p}(\xi)\sim0,
\end{equation}
which gives by integration the desired local behavior both as going out of $P_0$ (that is, \eqref{interm3} holds true in the limit $\xi\to0$) and as entering $P_0$ (that is, \eqref{interm3} holds true in the limit $\xi\to\xi_0\in(0,\infty)$). We omit here the details, full calculations being given in \cite[Section 2]{IS20b}.
\end{proof}

\noindent \textbf{Remark.} Lemma \ref{lem.P0} has as immediate consequence the proof of Part 1 in Theorem \ref{th.small}.

\medskip

The critical point $P_1$ is related to the interfaces of Type I. Its local behavior is not influenced in any ways by the dimension $N$, thus we only state here the result and refer the reader to \cite[Lemma 3.1]{IS20b} for a full proof.
\begin{lemma}[Local analysis near $P_1$]\label{lem.P1}
The system \eqref{PSsyst1} in a neighborhood of the critical point $P_1$ has a two-dimensional stable manifold and a one-dimensional unstable manifold. The profiles contained in the orbits entering $P_1$ have an interface of type I, with a local behavior given by
\begin{equation}\label{Beh.P1}
f(\xi)\sim\left[K-\frac{\beta(m-1)}{2m}\xi^2\right]^{1/(m-1)}, \quad K>0,
\end{equation}
for $\xi\to\xi_0=\sqrt{2mK/(m-1)\beta}\in(0,\infty)$.
\end{lemma}
We are only left with the local analysis near the critical point $P_2$, which will be one of the most important in the global behavior. Qualitatively, the behavior is similar to the one for dimension $N=1$, but at a more technical level the dependence on $N$ of the eigenvalues and eigenvectors will be important in the sequel, thus we provide a full proof.
\begin{lemma}[Local analysis near $P_2$]\label{lem.P2}
The system \eqref{PSsyst1} in a neighborhood of the critical point $P_2$ has a two-dimensional stable manifold fully contained in the invariant plane $\{Z=0\}$ and a one-dimensional unstable manifold. There exists a unique orbit going out of $P_2$ and the profiles contained in this orbit have a local behavior given by \eqref{beh.P2} as $\xi\to0$.
\end{lemma}
\begin{proof}
The linearization of the system \eqref{PSsyst1} near the critical point $P_2$ has the matrix
$$
M(P_2)=\left(
  \begin{array}{ccc}
    -\frac{m-1}{(mN-N+2)\alpha} & \frac{(m-1)^2}{2(mN-N+2)\alpha} & 0 \\
    1-\frac{N}{(mN-N+2)\alpha} & -\frac{2\beta(mN-N+2)+N(m-1)+4}{2(mN-N+2)\alpha} & -1 \\
    0 & 0 & \frac{\sigma(m-1)+2(p-1)}{2(mN-N+2)\alpha} \\
  \end{array}
\right)
$$
with eigenvalues $\lambda_1$, $\lambda_2$ and $\lambda_3$ such that
$$
\lambda_1+\lambda_2=-\frac{(N+2)(m-1)+2(mN-N+2)\beta+4}{2(mN-N+2)\alpha}<0, \ \lambda_1\lambda_2=\frac{m-1}{2(mN-N+2)\alpha^2}>0
$$
and
$$
\lambda_3=\frac{\sigma(m-1)+2(p-1)}{2(mN-N+2)\alpha}>0.
$$
We thus notice that $\lambda_1$ and $\lambda_2$ are either both negative, or they are complex conjugated with negative real parts, while their eigenvectors have the $Z$-component equal to zero. It follows that the stable manifold of $P_2$ is contained in the invariant plane $\{Z=0\}$. There exists an unique orbit going out of $P_2$ into the half-space $\{Z>0\}$ in the phase space, tangent to the eigenvector corresponding to $\lambda_3$, that is
\begin{equation}\label{interm4}
e_3(\sigma)=\left(\frac{2\alpha(m-1)^2(mN-N+1)}{D(\sigma)},\frac{2\alpha(mN-N+2)[(m-1)\sigma+2(m+p-2)]}{D(\sigma)},1\right),
\end{equation}
where
\begin{equation}\label{interm5}
\begin{split}
D(\sigma)&=-(m-1)^2\sigma^2-(m-1)[(m-1)N+2(m+2p-1)]\sigma\\
&-8(m-1)^2N-8[p(m+p-2)+m-1]<0.
\end{split}
\end{equation}
The local behavior \eqref{beh.P2} of the profiles contained in the orbits going out of $P_2$ is obtained from the fact that on these orbits we have
\begin{equation*}
X(\xi)=\frac{m}{\alpha}\frac{f^{m-1}(\xi)}{\xi^2}\sim\frac{m-1}{2(mN-N+2)\alpha},
\end{equation*}
which can only hold true in terms of profiles in the limit $\xi\to0$, as it is easy to check (full details in dimension $N=1$ are given in \cite[Lemma 3.2]{IS20b}).
\end{proof}

\subsection{New critical points when $m+p=2$}\label{subsec.crit}

In the critical case $m+p=2$, the system \eqref{PSsyst1} simplifies a bit in the equation for $\dot{Z}$ and some new critical points appear. We observe that by letting $X=0$, in this case we already have $\dot{Z}=0$ and $\dot{X}=0$, thus it only remains to vanish the right hand side in the second equation, thus finding the \emph{critical parabola}
\begin{equation}\label{crit.par}
-Y^2-\frac{\beta}{\alpha}Y-Z=0, \ \ {\rm that \ is} \ \ P_0^{\lambda}=\left(0,\lambda,-\lambda^2-\frac{\beta}{\alpha}\lambda\right), \ \lambda\in\left[-\frac{\beta}{\alpha},0\right],
\end{equation}
and this critical parabola replaces the critical point $P_1$ and partly the critical point $P_0$. The linearization of the system \eqref{PSsyst1} in a neighborhood of the critical points $P_0^{\lambda}$ has the matrix
$$
M(P_0^{\lambda})=\left(
                   \begin{array}{ccc}
                     (m-1)\lambda & 0 & 0 \\
                     1-N\lambda & -2\lambda-\frac{\beta}{\alpha} & -1 \\
                     -(\sigma-2)(\lambda^2+\frac{\beta}{\alpha}\lambda) & 0 & 0, \\
                   \end{array}
                 \right)
$$
with eigenvalues $l_1=(m-1)\lambda<0$, $l_2=-2\lambda-\frac{\beta}{\alpha}$ and $l_3=0$. We will keep the notation $P_0=P_0^0$ for simplicity. The local analysis of the orbits entering these critical points depends on the sign of the eigenvalue $l_2$ and is identical to the one performed in dimension $N=1$. We thus only state it here and we refer the reader to \cite[Lemma 2.2]{IS21b} for the proof.
\begin{lemma}[Local behavior near $P_0^{\lambda}$]\label{lem.P0l}
Let $m$, $p$ and $\sigma$ as in \eqref{range.exp} with $m+p=2$. We then have
\begin{itemize}
\item For $\lambda\in(-\beta/\alpha,-\beta/2\alpha)$ the system \eqref{PSsyst1} near the critical point $P_0^{\lambda}$ has a one-dimensional center manifold, a one-dimensional unstable manifold and a one-dimensional stable manifold. The center manifold is contained in the critical parabola \eqref{crit.par}, the unstable manifold is contained in the invariant plane $\{X=0\}$ and there is a unique orbit entering $P_0^{\lambda}$ from the half-space $\{X>0\}$.
\item For $\lambda=-\beta/2\alpha$ the system \eqref{PSsyst1} near the critical point $P_0^{\lambda}$ gives us an interesting example of a \emph{center-stable two-dimensional manifold}.
\item For $\lambda\in(-\beta/2\alpha,0)$ the system \eqref{PSsyst1} near the critical point $P_0^{\lambda}$ has a one-dimensional center manifold and a two-dimensional stable manifold. The center manifold is contained in the parabola \eqref{crit.par}.
\end{itemize}
In all the three cases, the orbits entering $P_0^{\lambda}$ on the stable (or center-stable) manifold contain profiles with an interface at some $\xi_0\in(0,\xi_{max}]$.
\end{lemma}
Recall that for $m+p=2$, the two types of interface behavior coincide, thus it no longer makes sense to talk about Type I or Type II.

\medskip

\noindent \textbf{Remark 1}. We notice that, when $m+p=2$, we have $\dot{Z}=(\sigma-2)XZ>0$, thus the component $Z$ is non-decreasing on any orbit of the system. Thus, the maximum value for $Z$ allowed for profiles with interface is attained at the peak of the parabola \eqref{crit.par}, namely
$$
Z=\frac{\beta^2}{4\alpha^2},
$$
which leads to $\xi=\xi_{max}$, taking into account that when $m+p=2$, the definition of $Z$ in \eqref{PSvar1} transforms into $Z=m\xi^{\sigma-2}/\alpha^2$. This proves that no profiles with interface at a point $\xi_0>\xi_{max}$ may exist.

\medskip

\noindent \textbf{Remark 2}. The critical point $P_0$ becomes now one of the endpoints of the parabola \eqref{crit.par} and there are no longer interesting orbits entering it, as it readily follows from the same monotonicity in $Z$ of the orbits. On the other hand, from the ancient elliptic sector in Lemma \ref{lem.P0}, it maintains the orbits going out of it, with the local behavior \eqref{beh.P0}, as it can be very easily seen from the proof of Lemma \ref{lem.P0}. Indeed, when $m+p=2$, the flow on the center manifold is given by the reduced system
\begin{equation}\label{interm2bis}
\left\{\begin{array}{ll}\dot{X}&=\frac{1}{\beta}X[X-(m-1)\alpha Z]+O(|(X,Z)|^3),\\
\dot{Z}&=\frac{2}{\beta}ZX+O(|(X,Z)|^3),\end{array}\right.
\end{equation}
which can be directly integrated to give that locally its orbits are tangent to the explicit family
$$
X=K\sqrt{Z}-(m-1)\alpha Z, \qquad K\in\real,
$$
leading to the expected behavior \eqref{beh.P0}. However, we have a result which is in some sense analogous for the case $m+p=2$ to the elliptic sector near $P_0$ in the former case $m+p>2$.
\begin{lemma}\label{lem.P0mp2}
Let $m$, $p$ and $\sigma$ as in \eqref{range.exp} and such that $m+p=2$. Then there exist orbits going out of $P_0$ and entering one of the critical points $P_0^{\lambda}$ for some $\lambda\in[-\beta/2\alpha,0)$.
\end{lemma}
\begin{proof}
The proof is totally similar to the one of \cite[Proposition 4.5]{IS21b}, by considering the plane
$$
aX+Z=c, \qquad a=\frac{3}{(m-1)\alpha}, \ c>0 \ {\rm small}
$$
as an upper barrier (in $Z$) for the flow of the system \eqref{PSsyst1} in a neighborhood of the center manifold $T(X,Z)$. Since the equations for $\dot{X}$, respectively $\dot{Z}$ in the system \eqref{PSsyst1} do not depend on $N$, the calculations are identical to the ones performed in the proof of \cite[Proposition 4.5]{IS21b} and we omit them here.
\end{proof}

We end this section with the following stability result that follows from purely theoretical considerations but it is very important in the sequel. It says that the points of the first half of the critical parabola \eqref{crit.par} act together as a "big attractor". This is proved in detail as \cite[Proposition 2.4]{IS21b}, thus we will only state it here.
\begin{proposition}\label{prop.att}
Let $m$, $p$ and $\sigma$ as in \eqref{range.exp} and such that $m+p=2$. The set of points $\mathcal{S}=\{P_0^{\lambda}: -\beta/2\alpha<\lambda<0\}$ is an asymptotically stable set for the system \eqref{PSsyst1}. Moreover, fixing $\lambda_{m}$ and $\lambda_M$ such that
$$
-\frac{\beta}{2\alpha}<\lambda_m<\lambda_M<0,
$$
then the set $\{P_0^{\lambda}: \lambda_m<\lambda<\lambda_M\}$ is an asymptotically stable set for the system \eqref{PSsyst1}.
\end{proposition}

\section{Local analysis at infinity}\label{sec.inf}

In this second preparatory section, we perform a study of the orbits in a neighborhood of the critical points at the infinity of the phase space. This analysis will introduce some significant differences with respect to dimension $N=1$. As usual when studying the infinity of a dynamical system, we pass to the Poincar\'e hypersphere according to the theory in \cite[Section 3.10]{Pe} introducing the new variables $(\overline{X},\overline{Y},\overline{Z},W)$ by
$$
X=\frac{\overline{X}}{W}, \ Y=\frac{\overline{Y}}{W}, \ Z=\frac{\overline{Z}}{W}
$$
and we obtain by standard theory \cite[Theorem 4, Section 3.10]{Pe} that the critical points at space infinity lie on the equator of the Poincar\'e hypersphere, hence at points
$(\overline{X},\overline{Y},\overline{Z},0)$ where $\overline{X}^2+\overline{Y}^2+\overline{Z}^2=1$ and the following system is fulfilled:
\begin{equation}\label{Poincare}
\left\{\begin{array}{ll}\overline{X}\overline{Y}((N-2)\overline{X}+m\overline{Y})=0,\\
\overline{X}\overline{Z}[\sigma\overline{X}-(1-p)\overline{Y}]=0,\\
\overline{Z}\overline{Y}[(m+p-1)\overline{Y}+(\sigma+N-2)\overline{X}]=0,\end{array}\right.
\end{equation}
Taking into account that we are considering only points with coordinates $\overline{X}\geq0$ and $\overline{Z}\geq0$, we find the following critical points at infinity (on the Poincar\'e hypersphere):
\begin{equation*}
\begin{split}
&Q_1=(1,0,0,0), \ \ Q_{2,3}=(0,\pm1,0,0), \ \ Q_4=(0,0,1,0), \\
&Q_5=\left(\frac{m}{\sqrt{(N-2)^2+m^2}},-\frac{N-2}{\sqrt{(N-2)^2+m^2}},0,0\right),
\end{split}
\end{equation*}
whose expression (noticeably, of the point $Q_5$) already shows that there is a kind of bifurcation at $N=2$, when $Q_5$ coincides with $Q_1$. We thus split the analysis with respect to the dimension.

\subsection{Local analysis at infinity in dimension $N\geq3$}\label{subsec.inf3}

As noticed above, the most important points for the behavior of the system at infinity are $Q_1$ and $Q_5$. In order to perform the local analysis near them, we use the topologically equivalent system given in part (a) of \cite[Theorem 5, Section 3.10]{Pe} which realizes a projection on the $X$ variable and thus translates these critical points into the finite part of the new phase space. This system becomes in our case
\begin{equation}\label{systinf1}
\left\{\begin{array}{ll}\dot{y}=-(N-2)y+w-my^2-\frac{\beta}{\alpha}yw-zw,\\
\dot{z}=\sigma z-(1-p)yz,\\
\dot{w}=2w-(m-1)yw,\end{array}\right.
\end{equation}
where
\begin{equation}\label{interm6}
y=\frac{Y}{X}, \qquad z=\frac{Z}{X}, \qquad w=\frac{1}{X},
\end{equation}
and the minus sign has been chosen in the general framework of \cite[Theorem 5, Section 3.10]{Pe} in order to match the direction of the flow, a fact that is noticed from the equation for $\dot{X}$ in the original system \eqref{PSsyst1}, leading to $\dot{X}<0$ in a neighborhood of $Q_1$, since $|X/Y|\to+\infty$ near this point. We thus readily see that $Q_1$ is mapped into the critical point $(y,z,w)=(0,0,0)$ and $Q_5$ into the critical point $(y,z,w)=(-(N-2)/m,0,0)$ of the system \eqref{systinf1}. We analyze now these points in detail.
\begin{lemma}[Local analysis near $Q_1$]\label{lem.Q1}
Let $N\geq3$. The system \eqref{systinf1} in a neihghborhood of $Q_1$ has a two-dimensional unstable manifold and a one-dimensional stable manifold. The orbits going out on the unstable manifold contain profiles with the local behavior \eqref{beh.Q1} as $\xi\to0$.
\end{lemma}
\begin{proof}
The linerization of the system \eqref{systinf1} in a neighborhood of $(y,z,w)=(0,0,0)$ has the matrix
$$
M(Q_1)=\left(
         \begin{array}{ccc}
           2-N & 0 & 1 \\
           0 & \sigma & 0 \\
           0 & 0 & 2 \\
         \end{array}
       \right),
$$
with two positive eigenvalues and one negative eigenvalue. The stable manifold is contained in the plane $\{w=0\}$, which is not interesting as it corresponds to $X=\infty$. The orbits going out on the unstable manifold satisfy locally that $dz/dw\sim\sigma z/2w$, which after integration reads $z\sim Cw^{\sigma/2}$, respectively $Z\sim CX^{1-\sigma/2}$, $C>0$. On the one hand, this leads to the local behavior given by
$$
\frac{m}{\alpha^2}\xi^{\sigma-2}f(\xi)^{m+p-2}\sim C\left(\frac{m}{\alpha}\right)^{(2-\sigma)/2}\xi^{\sigma-2}f(\xi)^{(m-1)(2-\sigma)/2}
$$
which implies $f(\xi)\sim C$ for some $C>0$, and this asymptotic behavior is taken as $\xi\to0$, since $X\to\infty$ at $Q_1$. On the other hand, we want to be even more precise and we find that in a neighborhood of $Q_1$ the following asymptotic approximation holds true:
$$
\frac{dy}{dw}\sim-\frac{N-2}{2}\frac{y}{w}+\frac{1}{2},
$$
which by integration leads to
$$
y\sim\frac{w}{N}+Kw^{(2-N)/2}, \qquad {\rm with} \ K=0,
$$
since we want to pass through the point $(y,w)=(0,0)$ and $(2-N)/2<0$. Thus $y\sim w/N$, hence $Y=y/w\sim 1/N$. This local approximation will be very important in the proofs in Section \ref{sec.large}, and we also deduce from this approximation that
$$
\frac{m}{(m-1)\alpha}\xi^{-1}(f^{m-1})'(\xi)\sim\frac{1}{N}, \qquad {\rm as} \ \xi\to0,
$$
which gives after an integration the local behavior \eqref{beh.Q1} as $\xi\to0$, as claimed.
\end{proof}
\begin{lemma}[Local analysis near $Q_5$]
Let $N\geq3$. The critical point $Q_5$ is an unstable node. The orbits going out of it into the finite part of the phase space contain profiles presenting a vertical asymptote at $\xi=0$, with the local behavior
\begin{equation}\label{beh.Q5}
f(\xi)\sim D\xi^{-(N-2)/m}, \qquad {\rm as} \ \xi\to0, \qquad D>0.
\end{equation}
\end{lemma}
\begin{proof}
This is now immediate, since the linearization of the system \eqref{systinf1} in a neighborhood of $Q_5$ has the matrix
$$
M(Q_5)=\left(
         \begin{array}{ccc}
           N-2 & 0 & 1 \\
           0 & \sigma+\frac{(1-p)(N-2)}{m} & 0 \\
           0 & 0 & 2+\frac{(m-1)(N-2)}{m} \\
         \end{array}
       \right)
$$
with three positive eigenvalues for $N\geq3$. The local behavior is given by the fact that $y\to-(N-2)/m$, which in terms of profiles translates (after undoing the changes of variables \eqref{systinf1} and \eqref{PSvar1}) into
$$
\frac{f'(\xi)}{f(\xi)}\sim-\frac{N-2}{m\xi},
$$
which gives the local behavior \eqref{beh.Q5}. Moreover, this behavior is taken in the limit as $\xi\to0$, since we have that on any orbit going out of $Q_5$, $Z/X\to0$, that is,
$$
\xi^{\sigma}f(\xi)^{p-1}\sim\xi^{\sigma+(1-p)(N-2)/m}\to0,
$$
whence $\xi\to0$, since the last exponent is positive for $0<p<1$.
\end{proof}
Let us remark here that this local behavior is in strong contrast with the one-dimensional case \cite{IS20b}, where no vertical asymptote was allowed and the critical point $Q_1$ was an unstable node, while $Q_5$ had only one interesting orbit. We explain better this significant change in the forthcoming Subsection \ref{subsec.inf2}.

Let us now pass to the local analysis of the critical points $Q_2$ and $Q_3$. These points are analyzed by a new change of variable to a system obtained by projecting on the $Y$ variable, according to Part (b) of \cite[Theorem 5, Section 3.10]{Pe}. This gives that the flow of the system \eqref{PSsyst1} near the points $Q_2$ and $Q_3$ is topologically equivalent to the flow near the origin $(x,z,w)=(0,0,0)$ of the system
\begin{equation}\label{systinf2}
\left\{\begin{array}{ll}\pm\dot{x}=-mx-(N-2)x^2-\frac{\beta}{\alpha}xw+x^2w-xzw,\\
\pm\dot{z}=-(m+p-1)z-\frac{\beta}{\alpha}zw-(N+\sigma-2)xz-z^2w+xzw,\\
\pm\dot{w}=-w-\frac{\beta}{\alpha}w^2+xw^2-Nxw-zw^2,\end{array}\right.
\end{equation}
thus we can state the following result.
\begin{lemma}[Local analysis near $Q_2$ and $Q_3$]\label{lem.Q23}
The critical point $Q_{2}=(0,1,0,0)$ is an unstable node. The orbits going out of $Q_2$ to the finite part of the phase space contain profiles $f(\xi)$ which vanish at some $\xi_0\in(0,\infty)$ in the sense that $f(\xi_0)=0$, $(f^m)'(\xi_0)>0$. The critical point $Q_3$ is a stable node. The orbits entering the point $Q_3$ from the finite part of the phase space contain profiles $f(\xi)$ which vanish at some $\xi_0\in(0,\infty)$ in the sense that $f(\xi_0)=0$, $(f^m)'(\xi_0)<0$.
\end{lemma}
Since the system \eqref{systinf2} does not depend in any essential way on the dimension $N$, the proof is completely similar to the one of \cite[Lemma 3.4]{IS20b}, to which we refer. Notice that the profiles in Lemma \ref{lem.Q23} \emph{have no interface}, since they do not fulfill the contact condition $(f^m)'(\xi_0)=0$ at the vanishing point. We end the analysis with the critical point $Q_4$, which can be totally discarded according to the next Lemma.
\begin{lemma}[No orbits at $Q_4$]\label{lem.Q4}
There are no profiles contained in the orbits connecting to the critical point $Q_4$.
\end{lemma}
The proof is rather long and uses the fact that on an orbit connecting to $Q_4$, one has $Z\to\infty$, $Z/X\to\infty$ and $Z/Y\to\infty$. This is translated via \eqref{PSvar1} in terms of profiles and then a contradiction is obtained to any possible behavior connecting $Q_4$ (that is, when $\xi\to0$, when $\xi\to\xi_0\in(0,\infty)$ and when $\xi\to\infty$) by eliminating these possibilities one by one using the equation for the profiles \eqref{SSODE}. Since the equation \eqref{SSODE} is used as a whole only when $\xi\to\infty$ and the profile $f(\xi)$ is decreasing on an interval $(R,\infty)$, and in this case the new term $(N-1)(f^m)'(\xi)/\xi$ is easily negligible, the proof is very similar to the one of \cite[Lemma 3.6]{IS20b} and we omit here the details.

\subsection{The special case $N=2$}\label{subsec.inf2}

As we have seen in the previous subsection, in dimension $N=2$ the critical points $Q_1$ and $Q_5$ coincide and their local analysis is different. We gather in the next result these changes.
\begin{lemma}\label{lem.Q15N2}
Let $N=2$. Then the critical point $Q_1=Q_5$ is a saddle-node in the sense of the theory in \cite[Section 3.4]{GH}. There exists a two-dimensional unstable manifold on which the orbits go out into the half-space $\{Y>0\}$ and these orbits contain good profiles with local behavior given by \eqref{beh.Q1} as $\xi\to0$. All the rest of the orbits go out into the half-space $\{Y<0\}$ and contain profiles with a vertical asymptote at $\xi=0$ given by
\begin{equation}\label{beh.Q12}
f(\xi)\sim D\left(-\ln\,\xi\right)^{1/m}, \qquad {\rm as} \ \xi\to0, \ D>0.
\end{equation}
\end{lemma}
\begin{proof}
We observe that for $N=2$, the linear term in the equation for $\dot{y}$ in the system \eqref{systinf1} vanishes. More precisely, if we let $\mu=2-N$, we are dealing at $\mu=0$ with a \emph{transcritical bifurcation} according to the definition introduced in \cite{S73} (see also \cite[Section 3.4]{GH}). It follows that at $N=2$ the critical point $Q_1=Q_5$ is a saddle-node with a matrix having eigenvalues $\lambda_1=0$, $\lambda_2=\sigma$ and $\lambda_3=2$. The two-dimensional unstable manifold tangent to the vector space spanned by the eigenvectors $e_2=(0,1,0)$ and $e_3=(0,0,1)$, that is the plane $\{y=0\}$, contains orbits inherited from the critical point $Q_1$ and their analysis is totally similar as the one in Lemma \ref{lem.Q1}, thus their local behavior is given by \eqref{beh.Q1}, as proved there.

All the other orbits, according to the theory in \cite[Section 3.4]{GH}, go out tangent to the direction of the eigenvector corresponding to the eigenvalue $\lambda=0$, that is the $y$-axis in the variables of the system \eqref{systinf1}. This means that $|y/w|\to\infty$, $|y/z|\to\infty$ on these orbits when approaching the point $Q_1$. The first equation in \eqref{systinf1} gives then locally that $\dot{y}\sim-my^2<0$, which proves that the orbits go out into the region $\{y<0\}$, which coincides with $\{Y<0\}$ in a neighborhood of $Q_1$, according to \eqref{interm6}. This also shows that, without absolute values, $Y=y/w\to-\infty$ and $Y/Z=y/z\to-\infty$ along these orbits. It only remains to establish that the profiles contained in these orbits have a local behavior given by \eqref{beh.Q12}. The next plan of the proof is to use Eq. \eqref{SSODE} and show that in a neighborhood of the point $Q_1$ and along these orbits, the first two terms dominate over the last three and thus the local behavior is given in a first order approximation by the combination of them. To this end, we first notice that
\begin{equation}\label{interm7}
\frac{\beta\xi f'(\xi)}{(N-1)(f^m)'(\xi)/\xi}=\frac{\beta}{m(N-1)}\xi^2f(\xi)^{1-m}=\frac{\beta}{m(N-1)}\frac{1}{X}\to0,
\end{equation}
since $X\to\infty$ when approaching $Q_1$. Next, we also discover that
\begin{equation}\label{interm8}
\frac{\xi^{\sigma}f(\xi)^p}{\alpha f(\xi)}=\frac{Z}{X}\to0,
\end{equation}
and finally,
\begin{equation}\label{interm9}
\frac{\alpha f(\xi)}{(N-1)(f^m)'(\xi)/\xi}=\frac{\alpha}{m(N-1)}\frac{1}{\xi^{-1}(f^{m-2}f')(\xi)}=\frac{1}{(N-1)Y}\to0,
\end{equation}
as we just proved above that $Y=y/w\to-\infty$ on these orbits. By gathering \eqref{interm7}, \eqref{interm8} and \eqref{interm9}, we find that in a small neighborhood of $Q_1$, the last three terms in Eq. \eqref{SSODE} are negligible with respect to the first two of them, and the local behavior is led in the first approximation by
$$
(f^m)''(\xi)+\frac{N-1}{\xi}(f^m)'(\xi)\sim 0,
$$
which gives the desired logarithmic behavior \eqref{beh.Q12}. Since $X\to\infty$ at $Q_1$, it is then obvious that this limit is taken as $\xi\to0$.
\end{proof}
This transcritical bifurcation of the system \eqref{systinf1} at $N=2$ explains better why the analysis at infinity of the phase space strongly departs with respect to the one in dimension $N=1$: if thinking at the dimension $N$ as a parameter of the system (that can be considered any positive real number), at $N=2$ the two critical points match and they interchange their behavior: if in $N=1$ the unstable node was $Q_1$, while $Q_5$ had just one interesting orbit going out \cite{IS20b}, for $N\geq3$ it is $Q_5$ the new unstable node and $Q_1$ remains as a saddle. We are now ready to begin the study of the connections in the phase space.

\section{Existence of good profiles with interface}\label{sec.exist}

This goal of this section is to complete the proof of Theorem \ref{th.exist}. Of course, when $m+p>2$ we already know that there exist orbits with interface of Type II for any $\sigma>0$ by Lemma \ref{lem.P0}, thus we will be interested in proving the existence of profiles with interface of Type I, while for $m+p=2$ we show that there are good profiles entering any of the critical points $P_0^{\lambda}$ with $-\beta/2\alpha\leq\lambda<0$. The strategy of all these proofs is to show that the limiting behavior on the stable manifold entering the corresponding critical points is given by orbits coming, on one side, from the node $Q_5$ and on the other side from the node $Q_2$, and then conclude that there are some orbits in between that must contain good profiles, using the topological "three-sets argument". Since particularly in this section many proofs do not differ much from the ones in dimension $N=1$, we will be rather brief.

\subsection{Good profiles with interface of Type I when $m+p>2$}

We employ an argument of \emph{backward shooting from the interface point} similar to the one used in \cite[Section 4]{IS20b}. The original system \eqref{PSsyst1} is not suitable for this, since all the profiles with interface of Type I are gathered into a single critical point $P_1$. We thus introduce a further change of variable in order to identify any profile uniquely with its interface point. We thus let
\begin{equation}\label{PSvar3}
\begin{split}
&U=X^{(m+p-2)/(m-1)}=\left(\frac{m}{\alpha}\right)^{(m+p-2)/(m-1)}\xi^{-2(m+p-2)/(m-1)}f(\xi)^{m+p-2}, \\
&V=\frac{Z}{U}=\frac{1}{\alpha}\left(\frac{m}{\alpha}\right)^{(1-p)/(m-1)}\xi^{[\sigma(m-1)+2(p-1)]/(m-1)}.
\end{split}
\end{equation}
and write the new system in variables $(U,Y,V)$
\begin{equation}\label{PSsyst3}
\left\{\begin{array}{ll}\dot{U}=\frac{m+p-2}{m-1}U[(m-1)Y-2U^{(m-1)/(m+p-2)}],\\
\dot{Y}=-Y^2-\frac{\beta}{\alpha}Y+U^{(m-1)/(m+p-2)}(1-NY)-UV,\\
\dot{V}=\frac{\sigma(m-1)+2(p-1)}{m-1}U^{(m-1)/(m+p-2)}V,\end{array}\right.
\end{equation}
Notice that $V$ is an elementary function (a power) of $\xi$, thus it is increasing along the trajectories of the system \eqref{PSsyst3}. On the other hand, the profiles having an interface at type I at some fixed point $\xi=\xi_0\in(0,\infty)$ are contained on trajectories entering the critical points
$$
P(v_0)=\left(0,-\frac{\beta}{\alpha},v_0\right), \qquad v_0=\frac{1}{\alpha}\left(\frac{m}{\alpha}\right)^{(1-p)/(m-1)}\xi_0^{[\sigma(m-1)+2(p-1)]/(m-1)}.
$$
The backward shooting strategy is contained in the next statement.
\begin{proposition}\label{prop.exist1}
Let $m$, $p$ and $\sigma$ as in \eqref{range.exp} such that $m+p>2$. For any critical point $P(v_0)$ with $v_0\in(0,\infty)$, there exists a unique orbit entering it. Moreover:
\begin{itemize}
\item The orbits entering points $P(v_0)$ with $v_0>0$ sufficiently small contain profiles that are decreasing for $\xi\in(0,v_0)$ and having a vertical asymptote at $\xi=0$, thus corresponding to orbits going out of $Q_5$ in the original system \eqref{PSsyst1}
\item The orbits entering points $P(v_0)$ with $v_0>0$ sufficiently large correspond to orbits going out of $Q_2$ in the original system \eqref{PSsyst1}.
\end{itemize}
\end{proposition}
\begin{proof}
Taking into account that $m-1>m+p-2$, the linearization of the system \eqref{PSsyst3} near a point $P(v_0)$ has the matrix
$$
M(v_0)=\left(
         \begin{array}{ccc}
           -\frac{\beta(m+p-2)}{\alpha} & 0 & 0 \\
           -v_0 & \frac{\beta}{\alpha} & 0 \\
           0 & 0 & 0 \\
         \end{array}
       \right),
$$
having eigenvalues $\lambda_1=-(m+p-2)\beta/\alpha<0$, $\lambda_2=\beta/\alpha>0$ and $\lambda_3=0$. Well known theoretical results on the center manifold theory such as \cite[Theorem 2.15, Chapter 9]{CH} and the Local Center Manifold Theorem \cite[Theorem 1, Section 2.10]{Pe} prove that all the one-dimensional center manifolds of $P(v_0)$ have to contain a segment of the invariant line $\{U=0, Y=-\beta/\alpha\}$. This implies that the center manifold is unique in a neighborhood of any $P(v_0)$ and by well-known results also the one-dimensional stable and unstable manifolds are unique. A standard argument (see for example \cite{dPS02} for details) gives that there exists only one orbit entering $P(v_0)$ from outside the invariant plane $\{U=0\}$, tangent to the eigenvector corresponding to $\lambda_1$, leading to the desired uniqueness.

\medskip

Consider now the plane $\{Y=-\beta/2\alpha\}$. The direction of the flow of the system \eqref{PSsyst3} on it is given by the sign of
$$
F(U,V)=\frac{\beta^2}{4\alpha^2}+U\left[\left(1+\frac{N\beta}{2\alpha}\right)U^{(1-p)/(m+p-2)}-V\right],
$$
thus an orbit can cross it coming from the half-space $\{Y>-\beta/2\alpha\}$ only in the region where $F(U,V)<0$, that is
\begin{equation}\label{interm10}
V\geq h(U):=\left(1+\frac{N\beta}{2\alpha}\right)U^{(1-p)/(m+p-2)}+\frac{\beta^2}{4\alpha^2U}.
\end{equation}
We deduce, by taking derivative in \eqref{interm10} with respect to $U$, that the function $h(U)$ attains a minimum at
$$
U_0=\left[\frac{(m+p-2)(m-p)^2}{2(\sigma+2)(1-p)(2\sigma+4+N(m-p))}\right]^{(m+p-2)/(m-1)}
$$
and if we let $V_0=h(U_0)$, we get that an orbit might reach a critical point $P(v_0)$ coming from the positive part $\{Y>0\}$ of the phase space (and thus in particular having to cross also the plane $\{Y=-\beta/2\alpha\}$) if and only if $v_0\geq V_0$. Thus, for any $v_0\in(0,V_0)$, the unique orbit entering $P(v_0)$ stays forever in the half-space $\{Y<-\beta/2\alpha\}$ and contains decreasing profiles. Since we proved in Lemma \ref{lem.Q1} that the orbits going out of $Q_1$ begin with $Y=1/N>0$ and we cannot have $\alpha$-limit cycles due to the monotonicity of $U$ and $V$ along the trajectories in the region $\{Y<-\beta/2\alpha\}$, we conclude that the orbits entering $P(v_0)$ with $v_0\in(0,V_0)$ come from the critical point $Q_5$ in the original system \eqref{PSsyst1}.

\medskip

The proof of the final statement in Proposition \ref{prop.exist1} is based on the analysis of the limit connection entering the critical point $P_1$ in the system \eqref{PSsyst1} and coming from the unstable node $Q_2$ inside the invariant plane $\{X=0\}$ and an argument of continuity. Since the analysis of the plane $\{X=0\}$ does not depend on $N$ at all, this proof is completely identical to the one of \cite[Proposition 4.2]{IS20b}, to which we refer the reader.
\end{proof}
\begin{proof}[Proof of Theorem \ref{th.exist}, first part]
Let $A\subseteq(0,\infty)$ be the set of points $v_0\in(0,\infty)$ such that the unique orbit entering the critical point $P(v_0)$ is decreasing and comes from $Q_5$. A standard argument of continuity gives that $A$ is an open set (since $Q_5$ is an unstable node) which contains an interval $(0,V_1)$ according to Proposition \ref{prop.exist1}. Let then $C$ be the set of $v_0\in(0,\infty)$ such that the unique orbit entering $P(v_0)$ comes from the unstable node $Q_2$. A similar argument gives that $C$ is nonempty and open, containing an interval of the form $(V_2,\infty)$ according to Proposition \ref{prop.exist1}. It thus follows that there exists a closed and nonempty set $B=\real\setminus(A\cup C)$ of elements $v_0\in(0,\infty)$ such that the orbit entering $P(v_0)$ comes from a different critical point or from an $\alpha$-limit set. An argument based on the compactness of such set \cite[Theorem 1, Section 3.2]{Pe} together with an analysis done at the level of the profiles completely similar to the one performed at the end of \cite[Section 5]{IS21c} remove this latter possibility. Thus, any orbit entering some critical point $P(v_0)$ with $v_0\in B$ comes from one of the critical points $P_0$, $P_2$ or $Q_1$ and contains good profiles.
\end{proof}

\subsection{Good profiles with interface when $m+p=2$}

This case is more technical, due to the existence of the critical parabola \eqref{crit.par}. However, most of the proofs are similar to the ones for dimension $N=1$, thus we will follow the structure of \cite[Section 3]{IS21b} and only give details at the points where (at least from the technical point of view) the calculations or arguments are influenced by the dimension $N$. The general idea is to analyze the stable manifold for the points $P_0^{\lambda}$ on the parabola \eqref{crit.par}.
\begin{proposition}\label{prop.far}
Let $m$, $p$ and $\sigma$ as in \eqref{range.exp} with $m+p=2$. The orbits entering the critical points $P_0^{\lambda}$ with $\lambda\in(-\beta/\alpha,-\beta/2\alpha)$ come from the critical point $Q_5$ at infinity. They contain profiles $f(\xi)$ that start with a vertical asymptote at $\xi=0$ and are decreasing on their positive part.
\end{proposition}
\begin{proof}
This is immediate along the lines of the proof of Proposition \ref{prop.exist1}, by calculating the direction of the flow on the plane $\{Y=-\beta/2\alpha\}$ and showing that this plane cannot be crossed when coming from the half-space $\{Y>-\beta/2\alpha\}$ unless when $Z>\beta^2/4\alpha^2$, which is the peak of the parabola \eqref{crit.par}. Since for $m+p=2$, the $Z$ component is non-decreasing along the trajectories, the orbits entering points $P_0^{\lambda}$ with $\lambda\in(-\beta/\alpha,-\beta/2\alpha)$ have to lie forever in the region $\{Y<-\beta/2\alpha\}$ and thus come from $Q_5$.
\end{proof}
We next consider the parabolic cylinder having the parabola \eqref{crit.par} as section, given by
\begin{equation}\label{cyl}
-Y^2-\frac{\beta}{\alpha}Y-Z=0.
\end{equation}
The next technical result shows that the cylinder \eqref{cyl} is a barrier for the trajectories in the phase space associated to the system \eqref{PSsyst1}.
\begin{lemma}\label{lem.cyl}
The trajectories of the system \eqref{PSsyst1} entering one of the points $P_0^{\lambda}$ with $\lambda\in[-\beta/2\alpha,0)$ from the interior of the parabolic cylinder \eqref{cyl} stay forever in the half-space $\{Y<0\}$. The same holds true for the trajectories entering one of the points $P_0^{\lambda}$ with $\lambda\in[-\beta/2\alpha,0)$ from the exterior of the cylinder \eqref{cyl} but which have previously crossed the parabolic cylinder at a point lying in the half-space $\{Z>X\}$. All these orbits go out of the unstable node $Q_5$ and contain decreasing profiles with vertical asymptote.
\end{lemma}
\begin{proof}
The proof is very similar to the one of \cite[Lemma 3.2]{IS21b}. The main argument in it, and the only one that might depend on the dimension $N$, relies on the direction of the flow of the system on the parabolic cylinder \eqref{cyl}. The normal direction to it is given by $\overline{n}=(0,-2Y-\beta/\alpha,-1)$ and the direction of the flow of the system \eqref{PSsyst1} on the cylinder depends on the sign of the expression
\begin{equation*}
\begin{split}
G(X,Y)&=\left(-2Y-\frac{\beta}{\alpha}\right)X(1-NY)+(\sigma-2)X\left(Y^2+\frac{\beta}{\alpha}Y\right)\\
&=X\left[(\sigma-2)Y^2+(\sigma-2)\frac{\beta}{\alpha}Y-2Y-\frac{\beta}{\alpha}+2NY^2+\frac{N\beta}{\alpha}Y\right]\\
&=X\left[(\sigma+2N-2)Y^2+\left((\sigma-2+N)\frac{\beta}{\alpha}-2\right)Y-\frac{\beta}{\alpha}\right]=Xh(Y)
\end{split}
\end{equation*}
where
$$
h(Y)=Y\left[(\sigma+2(N-1))Y+(\sigma-2+N)\frac{\beta}{\alpha}\right]-\left(2Y+\frac{\beta}{\alpha}\right).
$$
Since we are only interested on the right half of the cylinder, that is, with $Y\geq-\beta/2\alpha$, we get that
$$
(\sigma+2(N-1))Y+(\sigma-2+N)\frac{\beta}{\alpha}\geq\left[-\sigma-2(N-1)+2\sigma+2N-4\right]\frac{\beta}{2\alpha}=(\sigma-2)\frac{\beta}{2\alpha}>0,
$$
recalling that \eqref{range.exp} implies $\sigma>2$ when $m+p=2$. Since $2Y+\beta/\alpha\geq0$, it follows that $h(Y)<0$ for any $Y\in[-\beta/2\alpha,0)$. We thus obtain that the parabolic cylinder \eqref{cyl} cannot be crossed from right to left (in terms of $Y$). Since the flow on the plane $\{X=Z\}$ is not influenced by $N$, the rest of the proof is straightforward, on the lines of the proof of \cite[Lemma 3.2]{IS21b}.
\end{proof}
Using this lemma, one can show that there are good profiles entering the critical points $P_0^{\lambda}$. A special treatment is needed for the peak of the parabola, thus it will be left aside in the first result.
\begin{proposition}\label{prop.right}
Let $m$, $p$ and $\sigma$ as in \eqref{range.exp} and such that $m+p=2$. For any $\lambda\in(-\beta/2\alpha,0)$ there exists at least an orbit entering the critical point $P_0^{\lambda}$ and containing good profiles with interface.
\end{proposition}
\begin{proof}[Sketch of the proof]
The proof follows the same scheme as in dimension $N=1$ \cite[Proposition 3.3]{IS21b}, thus we will only sketch its ideas and give in detail the calculations where dimension $N$ appears. In a first step, we want to characterize locally the stable manifold of $P_0^{\lambda}$ through an approximation of its trajectories by a one-parameter family. To this end, we translate $P_0^{\lambda}$ to the origin by setting
\begin{equation}\label{PSvar1.bis}
Y_1=Y-\lambda, \qquad Z_1=Z+\lambda^2+\frac{\beta}{\alpha}\lambda,
\end{equation}
and obtain the system
\begin{equation}\label{PSsyst1.bis}
\left\{\begin{array}{ll}\dot{X}=(m-1)\lambda X+(m-1)XY_1-2X^2,\\
\dot{Y}_1=-Y_1^2-\frac{2\alpha\lambda+\beta}{\alpha}Y_1+(1-N\lambda)X-NXY_1-Z_1,\\
\dot{Z}_1=-\frac{(\sigma-2)\lambda(\alpha\lambda+\beta)}{\alpha}X+(\sigma-2)XZ_1,\end{array}\right.
\end{equation}
where our critical point becomes $(X,Y_1,Z_1)=(0,0,0)$. This system can be approximated in first order, in a neighborhood of the origin, as it follows:
\begin{equation}\label{interm11}
\frac{dZ_1}{dX}\sim-\frac{(\sigma-2)(\alpha\lambda+\beta)}{(m-1)\alpha}, \qquad {\rm that \ is} \ Z_1=-\frac{(\sigma-2)(\alpha\lambda+\beta)}{(m-1)\alpha}X+o(|X|),
\end{equation}
an approximation we can plug in the system \eqref{PSsyst1.bis} to get a reduced system in $(X,Y_1)$
\begin{equation}\label{interm12}
\left\{\begin{array}{ll}\dot{X}=(m-1)\lambda X+O(|(X,Y_1)|^2),\\
\dot{Y}_1=-\frac{2\alpha\lambda+\beta}{\alpha}Y_1+\left(1-N\lambda+\frac{(\sigma-2)(\alpha\lambda+\beta)}{(m-1)\alpha}\right)X+O(|(X,Y_1)|^2),\end{array}\right.
\end{equation}
By similar arguments as in the proof of \cite[Proposition 3.3]{IS21b}, based on an application of the Hartman-Grobman Theorem, we infer that the trajectories of the system \eqref{PSsyst1.bis} entering the origin are tangent to the following one-parameter family of trajectories obtained by integrating the approximating linear system \eqref{interm12}
\begin{equation}\label{interm13}
Y_1=KX^{-\frac{2}{m-1}-\frac{2}{(\sigma+2)\lambda}}-\frac{(\sigma+2)(mN-\sigma-N+2)\lambda-(3\sigma-2)(m-1)}{(m-1)[(\sigma+2)(m+1)\lambda+2(m-1)]}X,
\end{equation}
with $K\in\real$ arbitrary.

\medskip

The next step is to prove that, on the one hand, there exists $K_1\in\real$ such that the orbits entering $P_0^{\lambda}$ tangent to the trajectories \eqref{interm13} with $K\in(-\infty,K_1)$ remain forever in the half-space $\{Y<0\}$ (and consequently come from the unstable node $Q_5$), and on the other hand there exists $K_2\in\real$ such that the orbits entering $P_0^{\lambda}$ tangent to the trajectories \eqref{interm13} with $K\in(K_2,\infty)$ come from the unstable node $Q_2$. For the former of these assertions we make strong use of Lemma \ref{lem.cyl} by showing that the orbits entering tangent to trajectories with $K$ very negative have to pass through the interior of the parabolic cylinder \eqref{cyl}, while the latter of these assertions is based on a "tubular neighborhood construction" near the orbit connecting $Q_2$ to the critical point $P_0^{\lambda}$ inside the invariant plane $\{X=0\}$ (a plane that is totally the same as in dimension $N=1$). The rigorous constructions are done in great detail in Steps 3, respectively 4, of the proof of \cite[Proposition 3.3]{IS21b}, to which we refer in order to keep the presentation sufficiently short.

\medskip

Showing that there are orbits entering $P_0^{\lambda}$ containing good profiles (that is, coming out of one of the critical points $P_0$, $P_2$ or $Q_1$) is now an immediate application of the "three sets argument", since the intervals of $K$ in the trajectories tangent to the ones in \eqref{interm13} and coming from the unstable nodes $Q_5$ and $Q_2$ are two open sets. Taking $\overline{K}\in\real$ to be the supremum of all $K\in\real$ such that the trajectory tangent to the one in \eqref{interm13} for such $K$ comes from $Q_5$, we readily observe that the trajectory tangent to \eqref{interm13} for $K=\overline{K}$ does not come either from $Q_5$ or $Q_2$. It thus follows that this orbit begins either in one of the remaining critical points (all of them "good") or in an $\alpha$-limit set. But the possibility of the $\alpha$-limit set is impossible, as it follows from combining its compactness given in \cite[Theorem 1, Section 3.2]{Pe} with the strict monotonicity of the $Z$-component in the system \eqref{PSsyst1} outside the invariant planes $\{X=0\}$ and $\{Z=0\}$. It thus follows that there exists at least one parameter $\overline{K}\in\real$ such that the trajectory entering $P_0^{\lambda}$ corresponding to $K=\overline{K}$ in \eqref{interm13} comes from one of the critical points $P_0$, $P_2$ or $Q_1$, thus it contains good profiles with one of the possible behaviors \eqref{beh.P0}, \eqref{beh.P2} or \eqref{beh.Q1}.
\end{proof}
We end this section with the analysis, based on the same ideas but with further technical complications, of the center-stable manifold near the peak of the parabola, the point $P_0^{\lambda}$ with $\lambda=-\beta/2\alpha$.
\begin{proposition}\label{prop.peak}
Let $m$, $p$ and $\sigma$ as in \eqref{range.exp} and such that $m+p=2$. There exists at least an orbit entering the critical point $P_0^{-\beta/2\alpha}$ and containing good profiles with interface.
\end{proposition}
\begin{proof}[Sketch of the proof]
The proof follows the same ideas and plans as the one of the previous proposition. In fact, the theoretical arguments leading to the existence of a good orbit are exactly the same as in the proof of Proposition \ref{prop.right}, the main technical difference being done in the step of obtaining the approximate local development of the manifold as a one-parameter family of trajectories, which is more involved in this case. To this end, we let again the translation \eqref{PSvar1.bis} leading to the system \eqref{PSsyst1.bis}, with the very important difference that the linear term in the equation for $\dot{Y}_1$ vanishes since $2\alpha\lambda+\beta=0$. Setting also $X_1=X$ to unify notation, we need a further (linear) change of variables which is required in order to put the matrix of the linearization of the system \eqref{PSvar1.bis} into a normal (that is, Jordan) form. To this end, we set
\begin{equation}\label{PSvar2}
X_2=X_1, \ \ Y_2=CX_1+DY_1, \ \ Z_2=AX_1+BZ_1,
\end{equation}
where the coefficients are given by
$$
A=\frac{(\sigma-2)\beta^2}{\alpha^2}, \ B=\frac{2\beta(m-1)}{\alpha}, \ C=\frac{2\alpha(m-1)+\beta((m-1)N+\sigma-2)}{m-1}, \ D = (m-1)\beta.
$$
The new system in variables $(X_2,Y_2,Z_2)$ reads
\begin{equation}\label{PSsyst2}
\left\{\begin{array}{ll}\dot{X}_2=-\frac{(m-1)\beta}{2\alpha}X_2-\frac{2\alpha(m-1)+\beta((m-1)(N+2)+\sigma-2)}{\beta(m-1)}X_2^2+\frac{1}{\beta}X_2Y_2,\\
\dot{Y}_2=-\frac{\alpha}{2}Z_2-\frac{1}{(m-1)\beta}Y_2^2+EX_2Y_2-FX_2^2,\\
\dot{Z}_2=(\sigma-2)\left[\frac{\beta}{\alpha^2}X_2Y_2+X_2Z_2-\frac{\beta(2\alpha(m-1)+\beta((m-1)N+m\sigma-2))}{(m-1)\alpha^2}X_2^2\right],\end{array}\right.
\end{equation}
with coefficients
\begin{equation*}
\begin{split}
&E=\frac{2\alpha m^2+\beta\sigma(m+1)+2N\beta(m-1)-2(\alpha+\beta)}{(m-1)^2\beta}, \\ &F=\frac{[2\alpha m(m-1)+\beta((m-1)N+m\sigma+2(m^2-3m+1))][(2\alpha+N\beta)(m-1)+\beta(\sigma-2)]}{(m-1)^3\beta}.
\end{split}
\end{equation*}
The good news is that the complicated coefficients $E$ and $F$ will not be dominating. Indeed, going back to the local linear dependence \eqref{interm11} and setting there $\lambda=-\beta/2\alpha$, we readily get that in a neighborhood of the origin of the system \eqref{PSsyst1.bis} the following approximation holds true
$$
Z_1=-\frac{\sigma-2}{\sigma+2}X_1+o(X_1),
$$
whence we deduce that $Z_2=o(X_2)$ by replacing $Z_1$, $X_1$ as in \eqref{PSvar2} and taking into account the precise values of $A$ and $B$. Thus, in a similar manner as in dimension $N=1$ we get the same local approximation for the center manifolds near the critical point $P_0^{-\beta/2\alpha}$, namely
\begin{equation}\label{center.man2}
X_2=K\exp\left(-\frac{(m-1)^2\beta^2}{2\alpha Y_2}\right)+o\left(\exp\left(-\frac{(m-1)^2\beta^2}{2\alpha Y_2}\right)\right), \ K\geq0.
\end{equation}
The detailed calculations leading to the local asymptotic development \eqref{center.man2} are given in the proof of \cite[Proposition 3.4]{IS21b}, also showing there how the terms with coefficients $E$ and $F$ are negligible in the calculations coming from the Local Center Manifold Theorem.

\medskip

Once obtained the local approximation of the trajectories on the center-stable manifold of $P_0^{-\beta/2\alpha}$ in the form of a one-parameter family of approximating orbits, the rest of the proof follows completely similar lines to the final (and already standard) steps in the proof of Proposition \ref{prop.right} (see also \cite[Proposition 3.4]{IS21b} for more details).
\end{proof}
Theorem \ref{th.exist} in the case $m+p=2$ is now a simple consequence of the two Propositions \ref{prop.right} and \ref{prop.peak}. Finally, if $m+p<2$ it is quite easy to see that there are no longer interface behaviors, as done in \cite[Section 7]{IS20b}, either on our system or on some equivalent one, thus no good profiles with interface may exist.

\section{Classification of profiles for $\sigma$ small. Proof of Theorem \ref{th.small}}\label{sec.small}

This section is devoted to the proof of Theorem \ref{th.small} both when $m+p>2$ and when $m+p=2$. Let us first notice that \textbf{Part 1} in Theorem \ref{th.small} is already proved by the existence of the elliptic sector in Lemma \ref{lem.P0} for $m+p>2$ and Lemma \ref{lem.P0mp2} for $m+p=2$. The core of the rest of the arguments is to control the orbits going out of the critical point $P_2$, thus we begin with a very useful lemma gathering several facts related to these trajectories.
\begin{lemma}\label{lem.monot}
Let $m$, $p$ and $\sigma$ be as in \eqref{range.exp}. Then at any point lying on the orbit going out of $P_2$ we have $X<X(P_2)$, $Y<Y(P_2)$. Moreover, if $\sigma\geq2$, the component $X$ is decreasing and the component $Y$ is also decreasing in the half-space $\{Y\geq0\}$ along the trajectory going out of the point $P_2$.
\end{lemma}
\begin{proof}
We start from a very useful inequality
\begin{equation}\label{intermX}
\frac{1}{N}-Y(P_2)=\frac{1}{N}-\frac{1}{\alpha(mN-N+2)}=\frac{2(N(m-p)+\sigma+2)}{N(\sigma+2)(mN-N+2)}>0,
\end{equation}
hence $Y(P_2)<1/N$. Let us now consider the planes $\{X=X(P_2)\}$ and $\{Y=Y(P_2)\}$. The direction of the flow on the former of them is given by the sign of the expression
\begin{equation}\label{interm14}
X(P_2)[(m-1)Y-2X(P_2)]=(m-1)X(P_2)(Y-Y(P_2))<0, \ {\rm provided} \ Y<Y(P_2).
\end{equation}
The direction of the flow on the latter is given by the sign of the expression
\begin{equation}\label{interm15}
\begin{split}
h(X,Z)&=-Y(P_2)^2-\frac{\beta}{\alpha}Y(P_2)+X(1-NY(P_2))-Z\\&=X(1-NY(P_2))-\frac{(mN-N+2)\beta+1}{(mN-N+2)^2\alpha^2}-Z.
\end{split}
\end{equation}
We infer from \eqref{intermX} that $1-NY(P_2)>0$, hence the right hand side of \eqref{interm14} is increasing in $X$. Thus, for any $X<X(P_2)$ we have $h(X,Z)<h(X(P_2),Z)<0$, since
\begin{equation*}
h(X(P_2),Z)=\frac{(m-1)(N(m-p)+\sigma+2)}{\alpha(mN-N+2)^2(\sigma+2)}-\frac{(mN-N+2)\beta+1}{(mN-N+2)^2\alpha^2}-Z=-Z<0.
\end{equation*}
We thus get from \eqref{interm14} and \eqref{interm15} that the region $\{X<X(P_2), Y<Y(P_2)\}$ is invariant for the flow of the system \eqref{PSsyst1}. Since the orbit going out of $P_2$ starts decreasingly (as it readily follows from the direction of the eigenvector $e_3(\sigma)$ in \eqref{interm4}), it enters this region and thus stays forever inside it.

Let us now restrict to $\sigma\geq 2$. Since the equation for $\dot{X}$ in the system \eqref{PSsyst1} gives that $X$ decreases along any trajectory in the region $\{Y<0\}$ and both components $X$, $Y$ start in a decreasing way in a neighborhood of $P_2$, a change of monotonicity of the $X$-component might only occur while $Y>0$. Assume for contradiction that there exists a first point $\eta_1>0$ such that $\dot{X}(\eta_1)=0$, $X''(\eta_1)\geq0$. Then $X''(\eta_1)=(m-1)X(\eta_1)\dot{Y}(\eta_1)\geq0$, whence $\dot{Y}(\eta_1)\geq0$, which implies that the $Y$-component had to change monotonicity already at some first point $\eta_2\leq\eta_1$. That means $\dot{Y}(\eta_2)=0$ and $Y''(\eta_2)\geq0$. If $\eta_2=\eta_1$, since $\dot{X}(\eta_2)=\dot{Y}(\eta_2)$, we obtain after differentiating again the second equation in \eqref{PSsyst1} that
$$
Y''(\eta_2)=-\dot{Z}(\eta_2)=-Z(\eta_2)[(m+p-2)Y(\eta_2)+(\sigma-2)X(\eta_2)]<0
$$
and we reach a contradiction. If $\eta_2\in(0,\eta_1)$, then $\dot{X}(\eta_2)<0$ (since $\eta_1>\eta_2$ is the first point where $X$ ceases to be decreasing). Taking into account that along the orbit going out of $P_2$ we have $Y\leq Y(P_2)\leq1/N$ by \eqref{intermX}, we derive again that
$$
Y''(\eta_2)=\dot{X}(\eta_2)(1-NY(\eta_2))-\dot{Z}(\eta_2)<0,
$$
and we reach again a contradiction. Thus the $X$-component is decreasing along the trajectory, and one can easily check that the $Y$-component is decreasing too, by employing similar arguments.
\end{proof}

\subsection{Analysis for $m+p>2$}

We are now ready to prove \textbf{Part 2} of Theorem \ref{th.small}. To this end, we pass again to the variables $(U,Y,V)$ introduced in \eqref{PSvar3} and work with the system \eqref{PSsyst3}. Denote by $U(P_2)$ the $U$-coordinate of $P_2$ in these new variables.
\begin{proof}[Proof of Theorem \ref{th.small} when $m+p>2$]
We drive the orbits from $P_2$ and $P_0$ for $\sigma$ sufficiently small and prove that \emph{all of them} enter $P_0$. We know by Lemma \ref{lem.monot} that these orbits stay forever in the region $\{X<X(P_2), Y<Y(P_2)\}$. Consider now the plane $\{NY+kV=1\}$, that will be used as an upper barrier in the half-space $\{Y>0\}$, with $k>0$ to be chosen later. The flow of the system \eqref{PSsyst3} on this plane is given by the sign of the expression
\begin{equation*}
\begin{split}
F(U,Y,V)&=-N\left(Y^2+\frac{\beta}{\alpha}Y\right)+NU^{(m-1)/(m+p-2)}kV-NUV\\&+k\frac{\sigma(m-1)+2(p-1)}{m-1}U^{(m-1)/(m+p-2)}V
=-N\left(Y^2+\frac{\beta}{\alpha}Y\right)\\&+UV\left[k\left(N+\frac{\sigma(m-1)+2(p-1)}{m-1}\right)U^{(1-p)/(m+p-2)}-N\right]<0,
\end{split}
\end{equation*}
if we take $k$ such that
$$
\frac{1}{k}=\frac{(N+\sigma)(m-1)+2(p-1)}{N(m-1)}U(P_2)^{(1-p)/(m+p-2)}.
$$
Thus, the orbits going out of both $P_2$ and $P_0$ will remain below the plane $\{NY+kV=1\}$ and they will intersect the plane $\{Y=0\}$ at a point with $V\leq 1/k$. This in particular implies that
$$
UV\leq\frac{U(P_2)}{k}=\frac{[(N+\sigma)(m-1)+2(p-1)][\sigma(m-1)+2(p-1)]}{2N(\sigma+2)(mN-N+2)},
$$
which tends to zero as $\sigma$ approaches its lower bound $2(1-p)/(m-1)$. Thus, there exists $\sigma_0>2(1-p)/(m-1)$ such that for any $\sigma<\sigma_0$, we have
$$
UV<k_1=\frac{\beta^2}{4\alpha^2}
$$
at the point of crossing the plane $\{Y=0\}$. The rest of the proof is easy and completely similar to the proof of \cite[Proposition 5.1]{IS20b}: the region limited by the planes $\{Y=-\beta/2\alpha\}$, $\{Y=0\}$ and the hyperbolic cylinder $\{UV=k_1\}$ proves to be invariant for the flow, thus the orbits going out of both $P_0$ and $P_2$ will remain forever in this region and have to enter the critical point $P_0$ which is the only point in the closure of it. Since for any $\sigma\in(2(1-p)/(m-1),\sigma_0)$ there exist orbits containing good profiles and entering $P_1$, they have to come only from $Q_1$ and contain profiles behaving as \eqref{beh.Q1} as $\xi\to0$.
\end{proof}
We plot in Figure \ref{fig1} the orbits from $P_2$ and $P_0$ to illustrate visually the outcome of the previous analysis.
\begin{figure}[ht!]
  \begin{center}
  \includegraphics[width=11cm,height=8cm]{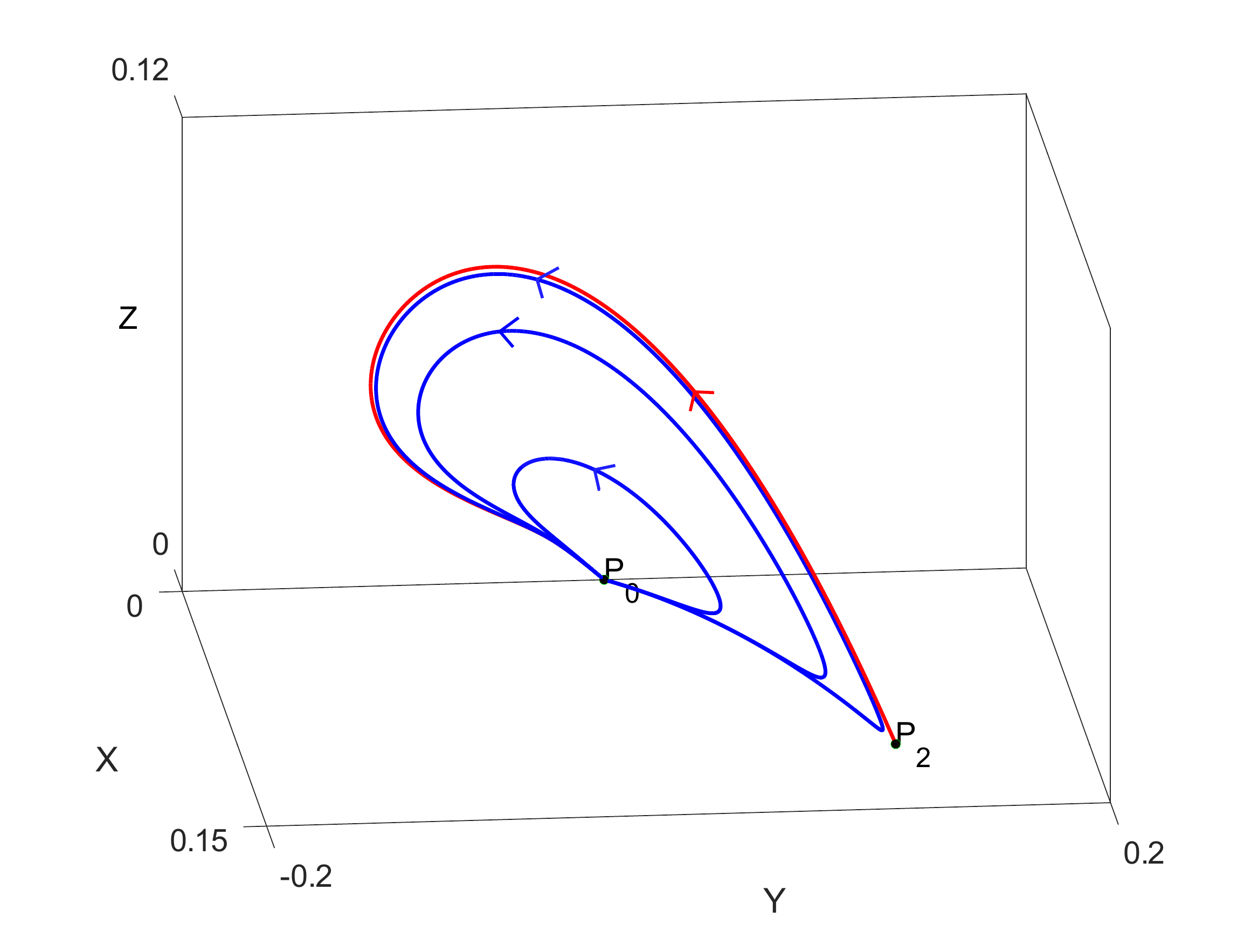}
  \end{center}
  \caption{Orbits from $P_0$ and $P_2$ in the phase space for $\sigma$ small. Experiment for $m=3$, $p=0.5$, $N=4$ and $\sigma=3.5$}\label{fig1}
\end{figure}

\subsection{Analysis for $m+p=2$}

The case of exponents $m$, $p$ and $\sigma$ as in \eqref{range.exp} with $m+p=2$ is more involved, due to the continuum of the target critical points on \eqref{crit.par}. The skeleton of the analysis and parts of the proofs follow the structure of the same analysis performed in dimension $N=1$ in \cite[Section 4]{IS21b}, thus we will only give here the general scheme and the details that depend on the dimension $N$. The core of the study is again given by the orbit going out of $P_2$.
\begin{proposition}\label{prop.P2mp2}
Let $m$, $p$ as in \eqref{range.exp} such that $m+p=2$. Then there exists $\sigma_0>2$ such that for any $\sigma\in(2,\sigma_0)$, all the orbits going out of $P_2$ and of $P_0$ enter one of the critical points $P_0^{\lambda}$ with $\lambda\in[-\beta/2\alpha,0)$.
\end{proposition}
\begin{proof}
We use the same system of regions $D_1$, $D_2$ and $D_3$ as in the proof of \cite[Propostition 4.2]{IS20b}.

\medskip

\noindent \textbf{Step 1.} Consider the plane
\begin{equation}\label{plane1mp2}
cY+Z=d, \qquad c=\frac{(m-1)^2}{(\sigma+2)^2}, \ d=\frac{(m-1)^2}{2(\sigma+2)^2}.
\end{equation}
The normal direction to it is given by the vector $(0,c,1)$ and the direction of the flow of the system \eqref{PSsyst1} on the plane is given by the sign of
\begin{equation}\label{interm19}
\begin{split}
H(X,Y)=&-\frac{(m-1)^2}{(\sigma+2)^2}Y^2-\frac{(m-1)^2(N+\sigma-2)}{(\sigma+2)^2}XY+\frac{\sigma(m-1)^2}{2(\sigma+2)^2}X\\
&-\frac{(2\sigma+5-m)(m-1)^3}{(\sigma+2)^4}Y-\frac{(m-1)^4}{2(\sigma+2)^4}.
\end{split}
\end{equation}
Fix now
$$
Y^*=-\frac{m-1}{6(2\sigma-m+5)}, \ {\rm that \ is} \ -\frac{(2\sigma+5-m)(m-1)^3}{(\sigma+2)^4}Y^*-\frac{(m-1)^4}{6(\sigma+2)^4}=0.
$$
We are generally interested in getting a negative sign for $H(X,Y)$. From the choice of $Y^*$ we readily get that
$$
-\frac{(2\sigma+5-m)(m-1)^3}{(\sigma+2)^4}Y-\frac{(m-1)^4}{6(\sigma+2)^4}<0, \ {\rm if} \ Y>Y^*,
$$
and we are left in \eqref{interm19} with a free term $-(m-1)^4/3(\sigma+2)^4$ to dominate the other two possibly positive ones. We split this term by fixing
$$
X^*=\frac{2(m-1)^2}{\sigma(\sigma+2)^2(N+2)}, \ {\rm that \ is} \ \frac{\sigma(m-1)^2}{2(\sigma+2)^2}X^*=\frac{1}{N+2}\frac{(m-1)^4}{(\sigma+2)^4}.
$$
With these choices for $X^*$ and $Y^*$, we restrict ourselves to the region $\mathcal{R}=\{0<X<X^*,Y>Y^*\}$. We are left with a single possibly positive term in \eqref{interm19} (the one in $XY$, if $Y<0$), thus we have to estimate
\begin{equation*}
\begin{split}
-\frac{(m-1)^2(N+\sigma-2)}{(\sigma+2)^2}XY&-\frac{(m-1)^4}{(\sigma+2)^4}\left(\frac{1}{3}-\frac{1}{N+2}\right)\\
&<-\frac{(m-1)^2(N+\sigma-2)}{(\sigma+2)^2}X^*Y^*-\frac{(m-1)^4}{(\sigma+2)^4}\frac{N-1}{3(N+2)}\\
&=\frac{2(m-1)^5(N+\sigma-2)}{6\sigma(\sigma+2)^4(N+2)(2\sigma-m+5)}-\frac{(m-1)^4}{(\sigma+2)^4}\frac{N-1}{3(N+2)}\\
&=\frac{(m-1)^4}{3(N+2)(\sigma+2)^4}\left[\frac{(m-1)(N+\sigma-2)}{\sigma(2\sigma-m+5)}-(N-1)\right],
\end{split}
\end{equation*}
and the term in brackets is negative for $\sigma$ sufficiently close to 2. Indeed, by taking limit as $\sigma\to2$ in this last term, by a simple calculation we find
$$
\frac{N(3m-19)+18-2m}{2(9-m)}\leq\frac{2(3m-19)+18-2m}{2(9-m)}=\frac{4(m-5)}{2(9-m)}<0,
$$
where we have strongly used the fact that $m\in(1,2)$, since $m=2-p$ in our case. Gathering all the previous estimates, we conclude that $H(X,Y)<0$ for any $(X,Y)\in\mathcal{R}$. Moreover, the plane \eqref{plane1mp2} does not intersect the parabolic cylinder \eqref{cyl} in the region $\{Y>Y^*\}$: such intersection would occur with $Y$-coordinate as one of the roots of
$$
P(Y)=Y^2+\left(\frac{2(m-1)}{\sigma+2}-\frac{(m-1)^2}{(\sigma+2)^2}\right)Y+\frac{(m-1)^2}{2(\sigma+2)^2}=0,
$$
and it is easy to check that $P(Y^*)>0$ and $P'(Y^*)>0$, the detailed calculations being exactly the same as in the proof of \cite[Proposition 4.2]{IS21b}.

\medskip

\noindent \textbf{Step 2.} Consider the same plane as we used in dimension $N=1$, namely
\begin{equation}\label{plane2mp2}
aX+Z=a, \qquad a=\frac{(m-1)^2(3\sigma+7-m)}{3(\sigma+2)^2(2\sigma+5-m)},
\end{equation}
and the direction of the flow of the system \eqref{PSsyst1} over this plane is given by
$$
L(X,Y)=\frac{(m-1)^2(3\sigma+7-m)}{3(\sigma+2)^2(2\sigma+5-m)}X\left[-\sigma X+(m-1)Y+\sigma-2\right]
$$
which is negative, provided that $-\sigma X+(m-1)Y+\sigma-2<0$. Taking the line $r_2$ as the intersection of \eqref{plane1mp2} with \eqref{plane2mp2}, this line crosses the plane $\{Y=0\}$ at a point lying in the region $\{X>X^*\}$. Indeed, calculating the intersection point, we find
$$
X-X^*=\frac{m-1}{2(3\sigma+7-m)}-\frac{2(m-1)^2}{(N+2)\sigma(\sigma+2)^2}\to\frac{(m-1)(m^2-14m+29+8N)}{16(13-m)(N+2)}, \ \ {\rm as} \ \sigma\to2,
$$
which implies that $X>X^*$ holds true for $\sigma\in(2,\sigma_0)$ for some $\sigma_0>2$. Moreover, the intersection of the plane \eqref{plane2mp2} with the parabolic cylinder \eqref{cyl} is given by a curve whose projection on the plane $\{Z=0\}$ is given by
\begin{equation}\label{interm20}
X=\frac{1}{a}\left[Y^2+\frac{2(m-1)}{\sigma+2}Y+b\right]=:g(Y).
\end{equation}

\medskip

\noindent \textbf{Step 3.} With the previous constructions and estimates, and since $Y^*$ does not depend on $N$, we are exactly in the same position as in the proof of \cite[Proposition 4.2]{IS21b}, thus we can construct the same system of invariant zones as there in order to conclude the proof. We just remind them here for the reader's convenience, and we refer to the above mentioned reference for details:
$$
D_1=\left\{0\leq X\leq X^*, 0\leq Y\leq\frac{1}{2}, 0\leq Z\leq-cY+d \right\},
$$
$$
D_2=\left\{0\leq X\leq X^*, eX-f\leq Y\leq0, -Y^2-\frac{2(m-1)}{\sigma+2}Y\leq Z\leq-cY+d \right\},
$$
where $e=(3\sigma+7-m)/3(2\sigma+5-m)$ and $f=-Y^*$, and finally
$$
D_3=\left\{0\leq X\leq X^*, g^{-1}(X)\leq Y\leq eX-f, -Y^2-\frac{2(m-1)}{\sigma+2}Y\leq Z\leq-aX+b\right\},
$$
where $g(Y)$ has been defined in \eqref{interm20}. For an easier understanding, we plot these regions in Figure \ref{fig2}. The previous calculations of the flow on the planes, together with the barrier represented by the parabolic cylinder \eqref{cyl} according to Lemma \ref{lem.cyl}, and with Lemma \ref{lem.monot} prove that for $\sigma\in(2,\sigma_0)$ sufficiently small (in order that the estimates in Steps 1 and 2 hold true), all the orbits going out of $P_2$ and $P_0$ enter the region $D_1$ and then pass to the region $D_2\cup D_3$ which is invariant. The monotonicity of components $X$ (when $Y<0$) and $Z$ over the orbits give that these orbits have to enter one of the critical points $P_0^{\lambda}$, as being the only ones in the closure of the invariant region $D_2\cup D_3$.
\end{proof}
\begin{figure}[ht!]
  \begin{center}
  \includegraphics[width=11cm,height=8cm]{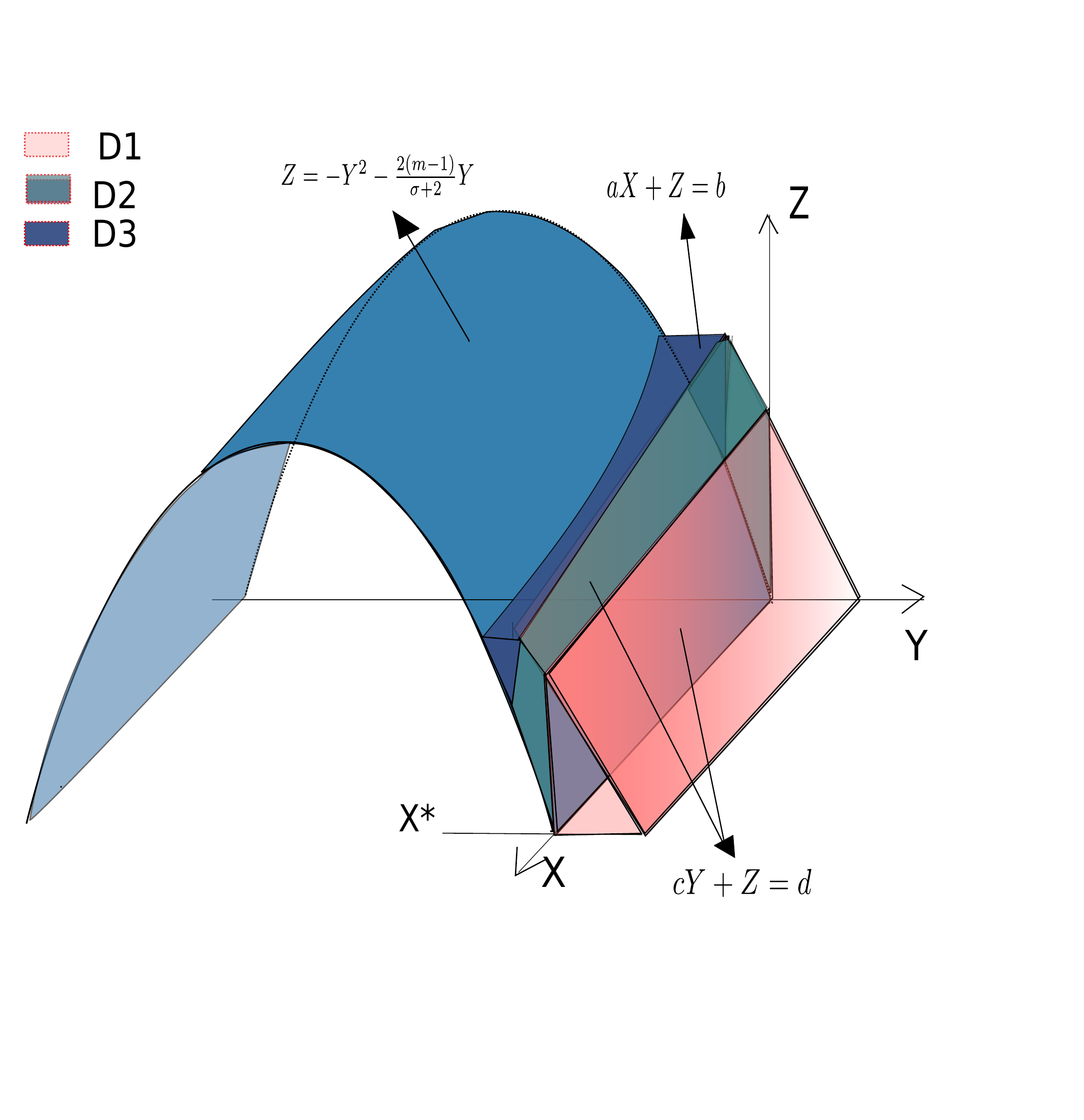}
  \end{center}
  \caption{A plot of the regions $D_1$, $D_2$ and $D_3$ in the phase space}\label{fig2}
\end{figure}

We also want to show that for $\sigma\in(2,\sigma_0)$ there exist good profiles contained in orbits starting from $Q_1$, in order to complete the proof of \textbf{Part 4} of Theorem \ref{th.small}. This follows readily from the next preparatory result.
\begin{lemma}\label{lem.Q1mp2}
The unique orbit going out of $Q_1$ inside the invariant plane $\{Z=0\}$ enters the critical point $P_2$.
\end{lemma}
\begin{proof}
The analysis in Lemma \ref{lem.P2} shows that, when restricted to the invariant plane $\{Z=0\}$, the critical point $P_2$ is stable (either a node or a focus according to the eigenvalues $\lambda_1$ and $\lambda_2$). Moreover, letting $Z=0$ and removing the last equation in \eqref{PSsyst1}, we are left with a system in $(X,Y)$ whose divergence of the vector field is
$$
{\rm div}\,F(X,Y)=(m-3)Y-(N+4)X-\frac{\beta}{\alpha}<0, \qquad {\rm for \ any} \ X, Y>0,
$$
since $m\in(1,2)$. We infer from Bendixon's criteria that this system has no limit cycles in the first quadrant of the plane $\{Z=0\}$. We also notice by an immediate inspection of the flow over the lines $X=0$, $Y=0$ and $Y=1/N$, that the strip $0\leq Y\leq 1/N$ is invariant for the orbits of the system contained in the plane $\{Z=0\}$, and the orbit going out of $Q_1$ begins with $Y=1/N$ and enters this strip, according to the analysis performed in Lemma \ref{lem.Q1}. Thus, this orbit stays forever in the strip $0\leq Y\leq 1/N$ and its $X$-component is obviously decreasing while $X>(m-1)/2N$, since $\dot{X}<0$. It thus follows that the unique orbit going out of $Q_1$ inside the plane $\{Z=0\}$ has to enter a critical point lying inside the rectangle formed by the lines $X=0$, $Y=0$, $Y=1/N$ and $X=(m-1)/2N$, and the only such critical point is $P_2$.
\end{proof}
Lemma \ref{lem.Q1mp2} together with a standard continuity argument gives that the orbits going out of $P_1$ very close to the plane $\{Z=0\}$ will also enter the region $D_1$ and then the system of regions used in the proof of Proposition \ref{prop.P2mp2}, thus ending up in one of the critical points $P_0^{\lambda}$, provided that $\sigma\in(2,\sigma_0)$. The proof of Theorem \ref{th.small}, \textbf{Part 4} is complete.

We end this section with the proof of \textbf{Part 5} of Theorem \ref{th.small}. To this end, assume that Part 3 has been already proved (it is postponed at the end of the paper), hence there exists a first value $\sigma^*$ such that the orbit going out of $P_2$ enters exactly the peak of the parabola \eqref{crit.par} for $\sigma=\sigma^*$ and enters some critical point $P_0^{\lambda}$ for any $\sigma\in(2,\sigma^*)$, as proved in Proposition \ref{prop.P2mp2}. Let $\lambda(\sigma)$ be the value of $\lambda$ such that the orbit going out of $P_2$ enters $P_0^{\lambda}$ for $\sigma\in(2,\sigma^*)$. Then, Part 5 of Theorem \ref{th.small} is an immediate consequence of
\begin{lemma}\label{lem.limit}
We have $\liminf\limits_{\sigma\to2}\lambda(\sigma)=0$.
\end{lemma}
\begin{proof}
The proof is very similar to the one of \cite[Lemma 4.4]{IS21b}, where it is proved a similar result in dimension $N=1$. Assume for contradiction that
$$
\liminf\limits_{\sigma\to2}\lambda(\sigma)=\lambda_0>0,
$$
which means in particular that the orbit going out of $P_2$ will not connect to the critical point $P_0^{\lambda_0/2}$ for $\sigma$ in a right-neighborhood of 2. Following the proof of \cite[Lemma 4.4]{IS21b}, we build an invariant region $D_0$ as being the solid limited by the planes $\{X=0\}$, $\{Y=0\}$, the surface of equation $\{\dot{Y}=0\}$ and the two-dimensional stable manifold of $P_0^{\lambda_0/2}$. In order for the proof to work exactly as in \cite[Lemma 4.4]{IS21b} (that is, to show that the region $D_0$ is indeed invariant), we need to study the flow of the system on the surface $\{\dot{Y}=0\}$, which depends now on $N$. The normal to this surface is given by
$$
\overline{n}=\left(1-NY,-2Y-\frac{\beta}{\alpha}-NX,-1\right)
$$
and the flow on the surface $\{\dot{Y}=0\}$ is given by the sign of the expression
$$
X\left[(1-NY)(m-1)Y-2(1-NY)X-(\sigma-2)Z\right]<0, \ \ {\rm for \ any} \ Y<0.
$$
Since the other "walls" of the solid $D_0$ form barriers for the flow once inside it (as it is immediate to check), it follows that an orbit entering $D_0$ cannot go out of this region afterwards (and thus has to connect to some critical point $P_0^{\lambda}$ with $\lambda\leq\lambda_0/2$, which produces a contradiction). The last step of the proof is to show that, for $\sigma$ sufficiently small, the orbit coming out of $P_2$ enters $D_0$ at the crossing point with the plane $\{Y=0\}$. But this follows from letting a plane of the form $\{NY+kZ=1\}$ as an upper barrier for the orbit from $P_2$ in the half-space $\{Y\geq0\}$ and thus show that it crosses the plane $\{Y=0\}$ at a point with $Z=1/k\to0$ as $\sigma\to2$. The calculation of the negative flow over the plane $\{NY+kZ=1\}$ has been already done in detail for $m+p>2$ in the current section and we do not repeat it here, as apart from the fact that $m+p=2$ simplifies a bit the expressions, the outcome is the same.
\end{proof}
We thus conclude that the orbit going out of $P_2$ visits all the critical points $P_0^{\lambda}$ with $\lambda\in(-\beta/2\alpha,0)$ one by one, as claimed in \textbf{Part 5} of Theorem \ref{th.small}.

\section{Classification of profiles for $\sigma$ large. Proof of Theorem \ref{th.large}}\label{sec.large}

Contrasting with the analysis in the previous Section \ref{sec.small}, the global analysis of the trajectories of the system \eqref{PSsyst1} with $\sigma$ sufficiently large do not depend on the sign of $m+p-2$. Moreover, we recall here that the geometric construction we use is new and the result is not only new in dimension $N\geq2$ but also improves the outcome of the similar analysis in dimension $N=1$ by completely discarding the possibility of connections to interface points from $Q_1$, a fact that we were not yet able to prove in the papers \cite{IS20b, IS21b} devoted to dimension $N=1$. Assume from the beginning that we work with, at least, $\sigma>2$, thus Lemma \ref{lem.monot} is in force.
\begin{proof}[Proof of Theorem \ref{th.large}]
The general idea of the proof is to construct a system of planes passing through $P_2$ and acting as barriers for the orbits coming from either $P_2$ or $Q_1$, thus limiting their access to the critical point $P_1$ (or to the critical parabola \eqref{crit.par}). We divide the (rather technical) proof into several steps for the readers' easiness.

\medskip

\noindent \textbf{Step 1. Plane of no return.} We show first that, given $\sigma>0$, there exists $Y_0>0$ sufficiently large such that, if $Y<-Y_0$, we have $\dot{Y}<0$ for $X<X(P_2)$. Indeed, since $Z\geq0$, it is enough to take $Y_0=(m-1)/2$, thus $-Y_0$ is the smallest root of the equation
$$
-Y^2-\frac{\beta}{\alpha}Y-NX(P_2)Y+X(P_2)=0.
$$
Then, for any $Y<-Y_0$ we have
$$
\dot{Y}=-Y^2-\frac{\beta}{\alpha}Y+X(1-NY)-Z<-Y^2-\frac{\beta}{\alpha}Y+X(P_2)(1-NY)<0.
$$
This implies that, once an orbit crossed the plane $\{Y=-Y_0\}$ at some point $X$ such that $X<X(P_2)$, it will remain forever in the region $\{Y<-Y_0\}$. Moreover, since
$$
-\frac{\beta}{\alpha}+Y_0=-\frac{\beta}{\alpha}+\frac{m-1}{2}=\frac{\sigma(m-1)+2(p-1)}{2(\sigma+2)}>0,
$$
it follows that $-Y_0<-\beta/\alpha$. Observe moreover that $Y_0$ is independent of $\sigma$.

\medskip

\noindent \textbf{Step 2. First plane through $P_2$.} Consider a plane of the form
\begin{equation}\label{plane1}
(\Pi_1) \qquad Y=\frac{Y(P_2)+B}{X(P_2)}X-B,
\end{equation}
where $B>0$ will be chosen later. The direction of the flow of the system \eqref{PSsyst1} on the plane $(\Pi_1)$ is given by the sign of the following complicated expression
\begin{equation}\label{interm16}
F(X,Z)=A_1X^2+A_2X+A_3+Z,
\end{equation}
with coefficients
\begin{equation*}
\begin{split}
&A_1=\frac{[K(\sigma+2)B+L][2m(\sigma+2)B+L]K}{(m-1)^2L^2},\\
&A_2=-\frac{2(\sigma+2)^2(m+1)KB^2+(\sigma+2)B[(K+2m)\sigma-K_1]+L^2}{(m-1)(\sigma+2)L},\\
&K_1=2(m-1)(m+2p-1)N+4(mp-2m+2p-1),\\
&A_3=\frac{B[(\sigma+2)B-m+p]}{\sigma+2},
\end{split}
\end{equation*}
where in order to simplify the notation, we set
$$
K=mN-N+2>0, \qquad L=\sigma(m-1)+2(p-1)>0.
$$
The previous equation describes a parabola $F(X,Z)=0$ inside the plane $(\Pi_1)$ with a positive peak (see Figure \ref{fig4}), where the direction of the flow towards the plane changes: the plane can be crossed by a trajectory arriving from bigger values of $X$ through the region interior to the parabola (that is, where $F(X,Z)<0$) and it cannot be crossed through the exterior of the parabola (that is, the region in $(\Pi_1)$ where $F(X,Z)>0$). This parabola has two intersections with the plane $Z=0$: one is at $X=X(P_2)$ (that is, it begins from the critical point $P_2$) and the second one can be also calculated explicitly in terms of $\sigma$
\begin{equation*}\label{interm17}
X_0(\sigma)=\frac{B(m-1)L(B\sigma+2B-m+p)}{[2Bm(\sigma+2)+L][(\sigma+2)KB+L]}.
\end{equation*}
Notice that, on the one hand, by geometrical considerations, $0<X_0(\sigma)<X(P_2)$ and on the other hand,
\begin{equation}\label{interm17}
X_0:=\lim\limits_{\sigma\to\infty}X_0(\sigma)=\frac{B^2(m-1)^2}{(2Bm+m-1)(BK+m-1)}>0.
\end{equation}
We thus make our choice for $B>0$ by letting it sufficiently large such that
\begin{equation}\label{interm18}
\frac{Y(P_2)+B}{X(P_2)}X_0-B<-2Y_0,
\end{equation}
where $Y_0=(m-1)/2$ is defined in Step 1 of the current proof. It follows that for $\sigma$ sufficiently large, we also get
$$
\frac{Y(P_2)+B}{X(P_2)}X_0(\sigma)-B<-Y_0.
$$

\medskip

\noindent \textbf{Step 3. Second plane through $P_2$.} Consider next a plane of the form
\begin{equation}\label{plane2}
(\Pi_2) \qquad Z=A(Y(P_2)-Y),
\end{equation}
with $A>0$ to be chosen later. We let, in order to simplify the notation,
\begin{equation}\label{intermhk}
h=Y-Y(P_2), \qquad k=X-X(P_2),
\end{equation}
and we find that the flow of the system \eqref{PSsyst1} over the plane $(\Pi_2)$ is given by the sign of the expression
$$
G(h,k)=\frac{2A(N(m-p)+\sigma+2)}{(\sigma+2)(mN-N+2)}(k+Mh)-A(N+\sigma-2)hk-A(m+p-1)h^2,
$$
where
\begin{equation*}
\begin{split}
&M=\frac{-(m-1)^2\sigma^2+[(m-1)(2A+1-m)N-4mp+4A+4p]\sigma+M_1}{4[N(m-p)+\sigma+2]},\\
&M_1=2(m-1)(2A+1-m)N-4p^2+8A-4m+4p+4.
\end{split}
\end{equation*}
Notice that, for $\sigma>0$ sufficiently large, we have $M<0$, which is very fine for our goals, since $h=Y-Y(P_2)\leq0$ in the region $Z\geq0$. Let us consider first $h=0$. Then $G(0,k)>0$ for any $k>0$, and it is easy to see that for $h<0$ and $k>0$ we also get $G(h,k)>0$ for $\sigma$ sufficiently large, since the negative contribution of $A(m+p-1)h^2$ is easily compensated by the positive contribution of the term containing $Mh$, which is also of order $\sigma$. We conclude that the plane $(\Pi_2)$ cannot be crossed from the right to the left in the region $\{X>X(P_2)\}$ by an orbit coming from the region $\{Z>A(Y(P_2)-Y)\}$ of the phase space. On the other hand, by collecting the terms of highest order in $\sigma$ in the expression of $G(h,k)$ (since we are interested in very large values of $\sigma$), the dominating terms with respect to $\sigma$ gather into
$$
G(h,k)\sim-Ah\sigma\left(k+\frac{(m-1)^2}{2(mN-N+2)}\right)=-Ah\sigma\left(X-X(P_2)+\frac{(m-1)^2}{2(mN-N+2)}\right).
$$
Noticing that
$$
\lim\limits_{\sigma\to\infty}\left[\frac{(m-1)^2}{2(mN-N+2)}-X(P_2)\right]=0,
$$
we find that $G(h,k)>0$ for $\sigma>0$ sufficiently large and $h<0$ if we keep $X$ bounded from below by a positive constant. We choose this constant to be $X_0$ introduced in \eqref{interm17}. On the contrary, it is obvious that $G(0,k)<0$ if $k<0$, that is $X<X(P_2)$. This is why we have to restrict ourselves to $h$ uniformly far from zero, thus, for any $\delta>0$, there exists a sufficiently large $\sigma=\sigma(\delta)$ (depending on $\delta$) such that the plane $(\Pi_2)$ cannot be crossed by any trajectory of the system through the region $\{X_0<X<X(P_2),Y<Y(P_2)-\delta\}$. Since these arguments are very geometric, a plot of the two planes and the parabola \eqref{interm16} where the direction of the flow changes is shown in Figure \ref{fig4}.
\begin{figure}[ht!]
  \begin{center}
  \includegraphics[width=11cm,height=8cm]{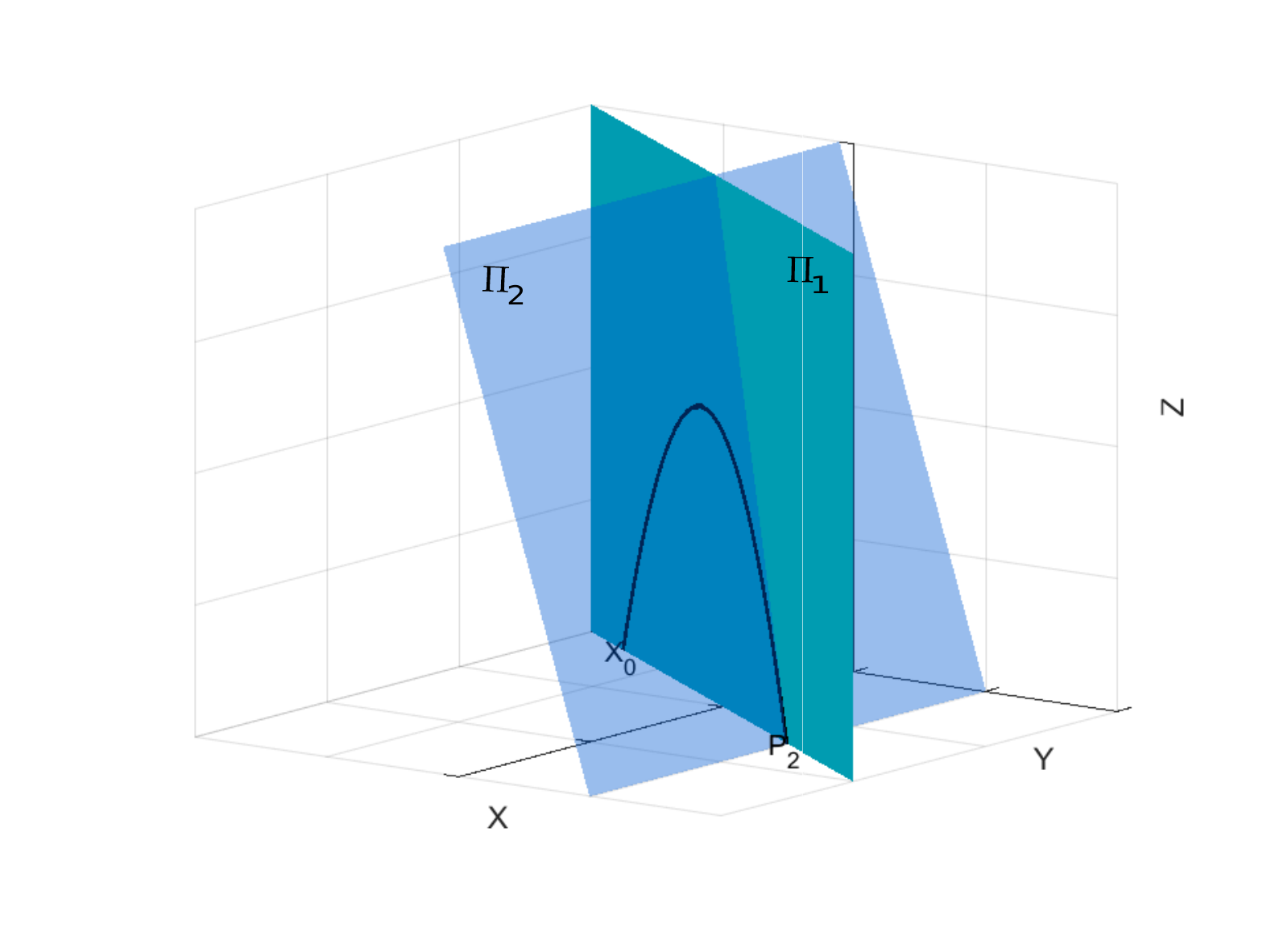}
  \end{center}
  \caption{The planes $(\Pi_1)$ and $(\Pi_2)$ in the phase space}\label{fig4}
\end{figure}

\medskip

\noindent \textbf{Step 4. Analysis in a neighborhood of $P_2$.} The previous analysis still has a flaw: we do not have uniformity in the magnitude of $\sigma$ since it depends on the choice of the small $\delta$ which measures how far we go to the left of the critical point $P_2$ (in terms of $Y$). Thus, if we fix a $\delta>0$ and then some $\sigma>0$ sufficiently large such that $G(h,k)>0$ for $X_0<X<X(P_2)$ and $Y<Y(P_2)-\delta$, we still have a small neighborhood of $P_2$ through which the orbits can cross our plane $(\Pi_2)$. In order to fix this problem, we have to use the plane $(\Pi_1)$ in our help. On the one hand, a simple inspection of the expression of $G(h,k)$ in the neighborhood of $P_2$ (that is, for $(h,k)$ defined by \eqref{intermhk} in a neighborhood of $(0,0)$) shows that the curve where the direction of the flow might change (and thus allowing some orbits to cross the plane $(\Pi_2)$) has the form
$$
k+Mh+{\rm lower \ order \ terms}=0,
$$
since the rest of the terms are quadratic in $(h,k)$. This is, in a sufficiently small neighborhood of $P_2$, a curve tangent to the line $k+Mh=0$, or equivalently
$$
Y=\frac{X(P_2)+MY(P_2)-X}{M}.
$$
On the other hand, the plane $(\Pi_1)$ intersects the plane $(\Pi_2)$ on a line of equation \eqref{plane1} on the plane $(\Pi_2)$. Taking the difference of the slopes of these two lines, we find that
$$
\lim\limits_{\sigma\to\infty}\left(\frac{Y(P_2)+B}{X(P_2)}-\frac{1}{M}\right)=\frac{2B(mN-N+2)+2(m-1)}{(m-1)^2}>0.
$$
It then follows that the line of intersection of the two planes comes first for $\sigma$ sufficiently large, and the difference in slopes is uniform. This proves that, letting $\sigma$ large enough, an orbit coming from the region $\{Z>A(Y(P_2)-Y)\}$ has to first intersect and cross the plane $(\Pi_1)$ in order to reach the small region that allows it to also cross $(\Pi_2)$. But we know that crossing the plane $(\Pi_1)$ is only allowed to be done in the region of the plane lying below the parabola \eqref{interm16}. Thus, the last "brick" of our system of planes is to finally choose the remaining free parameter $A$ sufficiently large such that the line of intersection between $(\Pi_1)$ and $(\Pi_2)$ is more "to the right" with respect to the parabola \eqref{interm16}. This is now easy from the geometrical point of view: the parabola \eqref{interm16} does not depend on $A$, while when $A$ tends to infinity, the equation $Z=A(Y(P_2)-Y)$ tends to a vertical line. Rigorously, by replacing
$$
X=\frac{(Y+B)X(P_2)}{B+Y(P_2)}
$$
into the parabola \eqref{interm16}, we get the expression of this parabola in terms of $Y$ and $Z$ and one can check (after rather tedious calculations, but which simplify when taking the limit as $\sigma\to\infty$) that for any choice of $A$ such that
$$
A>\frac{2(m-1)(mN-N+2)B^2+(m-1)(mN-N+2+2m)B+(m-1)^2}{B(mN-N+2)+(m-1)},
$$
and $\sigma$ sufficiently large, the line of intersection of the two planes lies "above" the parabola \eqref{interm16}. Thus, fixing $B$, then fixing $A$ according to the previous estimate and then taking $\sigma$ large enough in order that all the previous conditions to be satisfied, we infer that our system of planes can be only crossed by orbits in the region $\{0<X\leq X_0\}$, with $X_0$ defined in \eqref{interm17}.

\medskip

\noindent \textbf{Step 5. End of the proof.} Let us take now an orbit going out of $P_2$ for $\sigma>0$ sufficiently large such that the conditions in the previous steps are fulfilled. Then this orbit starts tangent to the eigenvector $e_3(\sigma)$ defined in \eqref{interm4}. We show that this orbit starts towards the region
\begin{equation}\label{region.large}
\left\{(X,Y,Z):Z>A(Y(P_2)-Y), \ Y<\frac{Y(P_2)+B}{X(P_2)}X-B\right\},
\end{equation}
(that is, "outside" the two planes) by taking scalar products of $e_3(\sigma)$ with the normal vectors to the planes. For the plane $(\Pi_2)$ we have
$$
\lim\limits_{\sigma\to\infty}e_3(\sigma)\cdot(0,A,1)=1,
$$
while for the plane $(\Pi_1)$, setting $C(\sigma)=(Y(P_2)+B)/X(P_2)$ in order to simplify the notation, we get
$$
e_3(\sigma)\cdot(C(\sigma),-1,0)=\frac{2(\sigma+2)(mN-N+2)[\sigma(m-1)-C(\sigma)(m-1)^2+2(m+p-2)]}{[\sigma(m-1)+2(p-1)]Q(\sigma)},
$$
which is positive provided $\sigma$ is sufficiently large, since
\begin{equation*}
\begin{split}
Q(\sigma)&=(m-1)^2\sigma^2+(m-1)(mN-N+2m+4p-2)\sigma+Q_0>0,\\
&Q_0=4(m-1)^2N+4(mp+p^2+m-2p-1),
\end{split}
\end{equation*}
and
$$
\lim\limits_{\sigma\to\infty}C(\sigma)=\frac{2[B(mN-N+2)+m-1]}{(m-1)^2}.
$$
It thus follows that the orbit coming out of $P_2$ enters the region \eqref{region.large}. On the other hand, we have shown in Lemma \ref{lem.Q1} that the orbits going out of $Q_1$ begin with $X=\infty$ and $Y=1/N$. Since $1/N>Y(P_2)$ according to \eqref{intermX}, it is obvious that these orbits also start in the region \eqref{region.large}. For $\sigma>0$ sufficiently large, all these orbits from both $P_2$ and $Q_1$ will remain in the region \eqref{region.large} at least until $X=X_0(\sigma)$, when they already entered the half-space $\{Y<-Y_0\}$ according to the choice of $B$. Thus these orbits will remain forever in the region $\{Y<-Y_0\}$ afterwards and cannot connect to either $P_1$ or the critical parabola \eqref{crit.par} in the case $m+p=2$. We then infer from Theorem \ref{th.exist} that the only good orbits for $\sigma$ large come from the point $P_0$, as stated.
\end{proof}

\noindent \textbf{Remark.} By the monotonicity of the coordinate $X$ in the half-space $\{Y<0\}$ and of the coordinate $Y$ in the region $\{Y<-Y_0\}$ (which comes from the choice of $Y_0$ in Step 1 of the above proof), all the orbits coming from the critical points $P_2$ and $Q_1$ for $\sigma\in(\sigma_1,\infty)$ sufficiently large cannot end up in a limit cycle, thus they all enter the critical point $Q_3$ at infinity. This is important for the final proof of this paper, the one of \textbf{Part 3} in Theorem \ref{th.small}.

\begin{figure}[ht!]
  \begin{center}
  \subfigure[Critical $\sigma^*$]{\includegraphics[width=7.5cm,height=6cm]{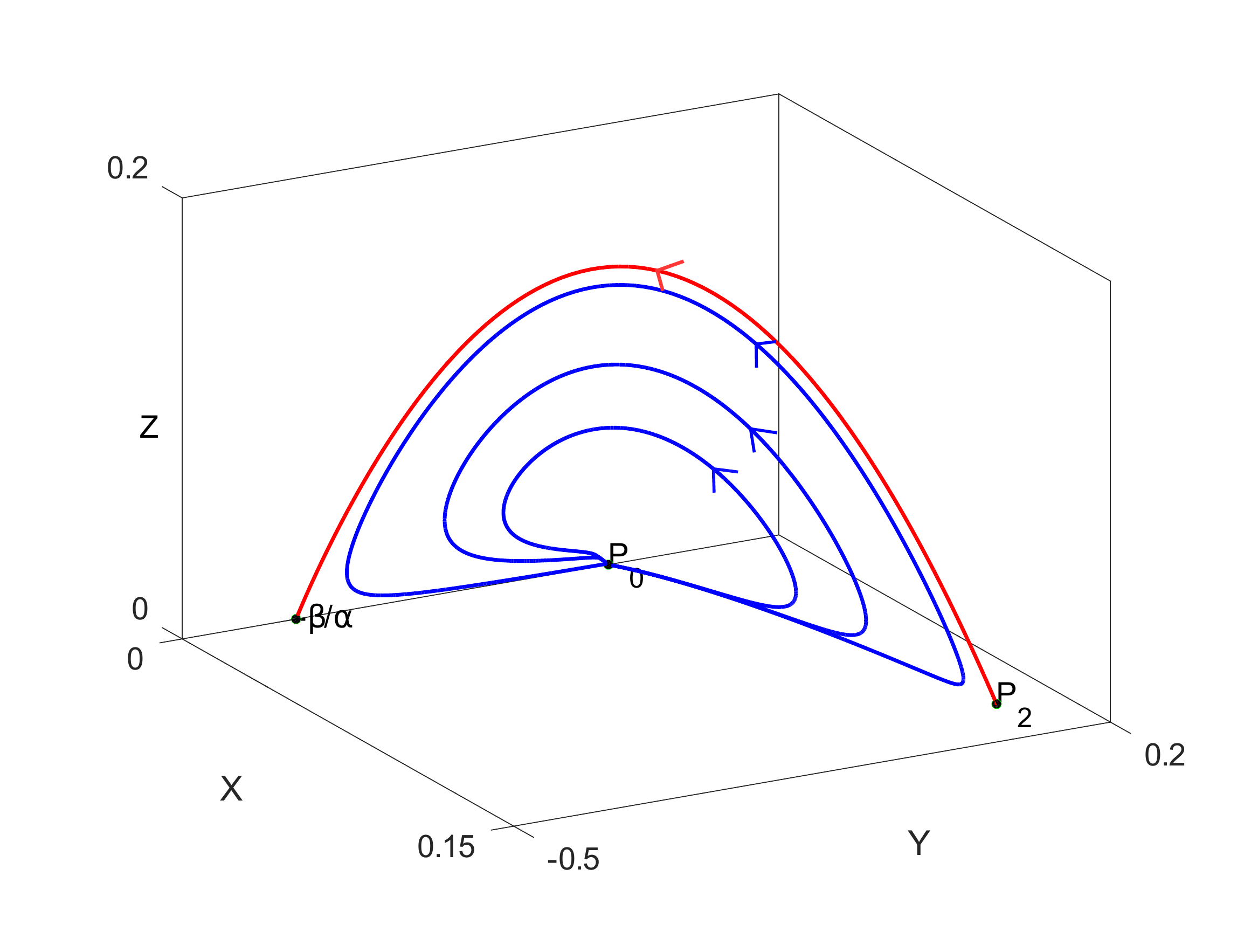}}
  \subfigure[$\sigma$ large]{\includegraphics[width=7.5cm,height=6cm]{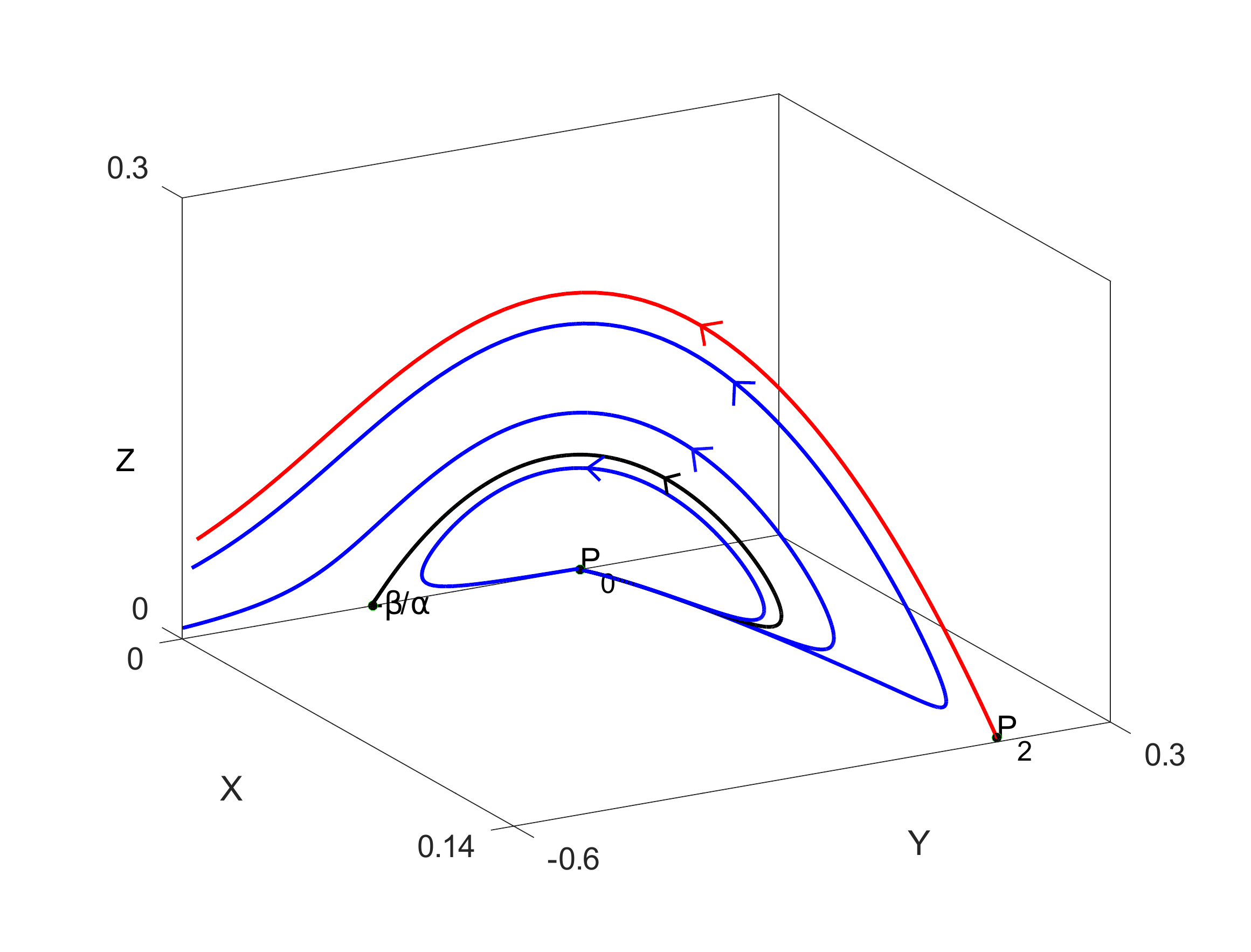}}
  \end{center}
  \caption{Orbits from $P_2$ and $P_0$ for different values of $\sigma$. Experiments for $m=3$, $p=0.5$, $N=4$ and $\sigma=4.822$, respectively $\sigma=6$}\label{fig3}
\end{figure}

\begin{proof}[Proof of Theorem \ref{th.small}, Part 3]
We use the standard "three-sets argument" by letting
\begin{equation*}
\begin{split}
&A:=\{\sigma>2(1-p)/(m-1): {\rm the \ orbit \ from \ }P_2 \ {\rm enters} \ P_0\},\\
&B:=\{\sigma>2(1-p)/(m-1): {\rm the \ orbit \ from \ }P_2 \ {\rm enters} \ P_1\},\\
&C:=\{\sigma>2(1-p)/(m-1): {\rm the \ orbit \ from \ }P_2 \ {\rm enters} \ Q_3\},
\end{split}
\end{equation*}
if $m+p>2$, with the obvious adaptation
\begin{equation*}
\begin{split}
&A:=\{\sigma>2: {\rm the \ orbit \ from \ }P_2 \ {\rm enters} \ P_0^{\lambda} \ {\rm with} \ \lambda\in(-\beta/2\alpha,0)\},\\
&B:=\{\sigma>2: {\rm the \ orbit \ from \ }P_2 \ {\rm enters} \ P_0^{-\beta/2\alpha}\},
\end{split}
\end{equation*}
in the case $m+p=2$. The sets $A$ and $C$ are both nonempty, as it follows from Proposition \ref{prop.P2mp2}, the previous Remark and the proof of Theorem \ref{th.small}, Part 2. Moreover, set $C$ is open by a standard continuity argument, since $Q_3$ is a stable node. In the case $m+p=2$, we infer from Lemma \ref{prop.att} that the whole set of points $P_0^{\lambda}$ with $\lambda\in(-\beta/2\alpha,0)$ is asymptotically stable, thus $A$ is also an open set. For the case $m+p>2$ things are a bit more involved, since the critical point $P_0$ is not an attractor, but nevertheless one can show that it behaves as an attractor on half-balls inside the half-space $\{Y<0\}$, the details being given at the end of the paper \cite{IS20b}. Thus, set $A$ is also open if $m+p>2$. We conclude that the set $B$ is closed and non-empty, thus it contains at least a value of $\sigma$, that we call $\sigma^*$. This, together with the definition of the set $B$, ends the proof.
\end{proof}
We plot in Figure \ref{fig3} the phase space for $\sigma=\sigma^*$ (when the orbit going out of $P_2$ enters $P_1$) and for $\sigma$ large, illustrating the outcome of Theorem \ref{th.large}.

\bigskip

\noindent \textbf{Acknowledgements} R. I. and A. S. are partially supported by the Spanish project PID2020-115273GB-I00. A. I. M. is partially supported by the Spanish project RTI2018-098743-B-100.

\bibliographystyle{plain}

\end{document}